\documentclass[reqno]{amsart}
\usepackage{amsthm}
\usepackage{amsmath}%
\usepackage{amsfonts}%
\usepackage{amssymb}%
\usepackage[foot]{amsaddr}
\usepackage{graphicx}
\usepackage{mathabx}
\usepackage{enumitem}
\usepackage{bbm}
\usepackage{lipsum}
\usepackage{tikz}
\usepackage{epsfig}
\usepackage{comment}

\usepackage{hyperref}

\newtheorem{theorem}{Theorem}
\theoremstyle{plain}

\newtheorem{definition}{Definition}
\newtheorem{example}{Example}

\newtheorem{lemma}{Lemma}

\newtheorem{proposition}{Proposition}

\numberwithin{equation}{section}

\newcommand{\sgn}{\operatorname{sgn}}

\renewcommand{\d}[1]{\mathinner{\mathrm{d}{#1}}}

\usepackage{geometry,mathtools}
\begin{document}
\title{Entropy Dissipation at the Junction for Macroscopic Traffic Flow models}
\author{Yannick Holle$^\dagger$}
\address{$^\dagger$Institut f\"ur Mathematik \\ RWTH Aachen University \\ Templergraben 55, 52062 Aachen, Germany}

\email[Y.~Holle]{holle@eddy.rwth-aachen.de}%
\date{May 31, 2021}

\begin{abstract}
	A maximum entropy dissipation problem at a traffic junction and the corresponding coupling condition are studied. We prove that this problem is equivalent to a coupling condition introduced by Holden and Risebro. An $L^1$--contraction property of the coupling condition and uniqueness of solutions to the Cauchy problem are proven. Existence is obtained by a kinetic approximation of Bhatnagar--Gross--Krook--type together with a kinetic coupling condition obtained by a kinetic maximum entropy dissipation problem. The arguments do not require $TV$--bounds on the initial data compared to previous results.
	We also discuss the role of the entropies involved in the macroscopic coupling condition at the traffic junction by studying an example.
\end{abstract}

\thanks{\textbf{Funding:} This work has been funded by the Deutsche Forschungsgemeinschaft (DFG, German Research Foundation) projects 320021702/GRK2326 Energy, Entropy, and Dissipative Dynamics (EDDy) and HE5386/22-1.}



\maketitle

\section{Introduction}

By hyperbolic conservation laws on networks we understand models on finite, directed graphs where the dynamics on each arc are modelled by hyperbolic conservation laws in one spatial dimension. The arcs correspond to the traffic roads supplemented with suitable hyperbolic conservation laws for vehicular traffic flow. At the vertices (called junctions) suitable conditions are needed to couple the solutions.
Hyperbolic conservation laws on networks were intensively studied in various directions (analysis, numerics, control,...) in the last two decades, see e.g. \cite{BCGHP2014}. See also \cite{Br2000,Da2016} for an introduction to hyperbolic conservation laws. 

We are interested in network models for macroscopic vehicular traffic flow. The scalar Lighthill-Whitham-Richards model (LWR model) \cite{LiWh1955,Ri1956} describes the dynamics of the vehicles on the roads. The LWR model on networks and different coupling conditions were studied by many authors, see e.g. \cite{ACD2017,BrNo2017,CGP2005,CoGa2010,DGP2018,FMR2021,HoRi1995,Le1993}. For a general introduction to traffic flow on networks, see \cite{GHP2016,GaPi2006}. 
A major problem is the selection of suitable coupling conditions at the junctions. The aim is to select coupling conditions modelling the physical dynamics at the traffic junction. 

In the seminal paper by Holden and Risebro \cite{HoRi1995} a coupling condition for the LWR model is defined by an optimization problem. Existence and uniqueness of solutions to the Riemann problem is obtained. For initial data with bounded total variation existence of solutions to the Cauchy problem is proven. We study this condition by using a maximum entropy dissipation problem which is equivalent to the optimization problem introduced by Holden and Risebro. We prove existence and uniqueness of solutions to the Cauchy problem without using a bound on the total variation of the initial data. 
Coclite and Garavello \cite{CoGa2010} introduce a vanishing viscosity approach for the LWR model on networks and Andreianov, Coclite and Donadello \cite{ACD2017} prove that the limit of the vanishing viscosity solutions satisfies an $L^1$--contractive coupling condition at the junction. Recently, Fjordholm, Musch and Risebro \cite{FMR2021} studied analytical and numerical aspects of $L^1$--contractive coupling conditions by using stationary solutions. In this paper we will use again the $L^1$--contraction property and stationary solutions to study a kinetic approximation for the LWR model.

We use a kinetic Bhatnagar--Gross--Krook--type (BGK) model and a kinetic coupling condition to approximate the LWR model. Kinetic models describing the dynamics at the junction were studied for different applications, e.g. traffic flow, chemotaxis and gas dynamics \cite{BKKP2016,BoKl2018,BoKl2018Scalar,BKKM2014,HeMo2009,HHW2020}. A typical approach starts by the selection of a kinetic coupling condition which is defined based on reasonable physical assumptions. In the next step the convergence towards a macroscopic solution is studied \cite{BKKM2014,HeMo2009}. In \cite{HHW2020} a BGK model for isentropic gas dynamics is coupled at the junction by a kinetic coupling condition. The convergence of the kinetic solutions is justified by compensated compactness and a macroscopic coupling condition is derived. We follow \cite{HHW2020} but since the LWR model is a scalar conservation law, we obtain stronger analytical results.

The LWR model is a scalar conservation law 
\begin{equation*}
	\partial_t\rho+\partial_x f(\rho)=0 \qquad  t>0,x\in \Omega,
\end{equation*}  
where $\rho=\rho(t,x)$ denotes the car density and $f$ the flux function. An assumption on the flux function typically used in the literature is that $f\colon [a,b]\to [0,\infty)$ is a sufficiently regular, concave function with $f(a)=f(b)=0$.  

We use the framework in \cite{GHP2016} and consider a single junction with $i=1,\dots,n$ incoming roads and $j=n+1,\dots,n+m$ outgoing roads. The roads are assumed to be infinitely long and the spatial variable $x$ is in the domains $\Omega_i=(-\infty,0)$ for $i=1,\dots,n$ and $\Omega_j=(0,\infty)$ for $j=n+1,\dots,n+m$. Throughout this paper we consider weak entropy solutions to 
\begin{equation}\label{eq:LWRmodel}
	\partial_t\rho_h+\partial_xf_h(\rho_h)=0 \qquad t> 0,x\in\Omega_h,h=1,\dots,n+m,
\end{equation} 
since it is well-known that the conservation law \eqref{eq:LWRmodel} does not admit classical solutions in general. 

We make the following assumptions on $f_h$:
\begin{enumerate}[label={\bf{(f.\arabic*)}}]
	\item \label{eq:Assumptionsfh1}$f_h\colon K_h\to [0,\infty)$ is Lipschitz continuous with $K_h=[a_h,b_h]$;
	\item \label{eq:Assumptionsfh2}$f_h$ is concave with $f_h(a_h)=f_h(b_h)=0$;
	\item \label{eq:Assumptionsfh3}
		\begin{equation*}
			\sum_{i=1}^{n}f_i(0)-\sum_{j=n+1}^{n+m}f_j(0)=0;
		\end{equation*}
	\item \label{eq:Assumptionsfh4} $f_h'\colon K_h\to \mathbb R$ is not constant on any non-trivial subinterval of $K_h$. 
\end{enumerate}
Notice that the non-degeneracy condition \ref{eq:Assumptionsfh4} is only used to pass to the limit in the kinetic approximation. For notational convenience we extend $f_h$ by $f_h(\rho)=0$ for $\rho\in\mathbb R\backslash K_h$.

We impose the initial data $\rho_{0,h}\in L^1(\Omega_h,K_h)$ by  
\begin{equation}
\rho_h(0,x):=\lim_{t\to 0+}\rho_h(t,x)=\rho_{0,h}(x)\qquad\text{for a.e. }x\in\Omega_h,
\end{equation}
and assume that the above limit is attained in the strong $L^1_\mathrm{loc}$--sense.

It remains to define a coupling condition at the junction.
First, we study the initial value problem with constant initial data on each road, called generalized Riemann problem \cite{GHP2016}. We assume that solutions to the generalized Riemann problem coincide with solutions to standard Riemann problems restricted to $x\in \Omega_h$.

We recall the definition of solutions to the generalized Riemann problem introduced by Holden and Risebro \cite{HoRi1995}. 
\begin{definition}\label{def:HoldenRisebro}
	Fix constant initial data $\rho_{0,h}\in K_h$ and concave functions $\hat g_{h}\colon [0,1]\to \mathbb R$.
	A solution to the generalized Riemann problem in the sense of Holden and Risebro \cite{HoRi1995} is given by the restrictions of $(\rho_1,\dots,\rho_{n+m})$ to $x\in \Omega_h$, 
	where $\rho_h\colon [0,\infty)_t\times \mathbb R_x\to K_h$ are the weak entropy solutions to the standard Riemann problems with initial data
	\begin{equation*}
		\rho_h(0,x)=
		\begin{cases}
			\rho_{0,h}\quad &\text{if }x\in\Omega_h,\\
			\tilde\rho_h\quad &\text{if }x\notin\Omega_h,
		\end{cases}
	\end{equation*}
	solving 
	\begin{alignat*}{3} 
		\underset{\tilde\rho_h\in K_h}{\max} & \quad & & \sum_{h=1}^{n+m} \hat g_{h}\!\left(\frac{f_h(\rho_h(t,0))}{f_h^{\max}}\right)\\
		\text{s.t.}&&& \sum_{i=1}^n f_i(\rho_i(t,0)) -\sum_{j=n+1}^{n+m} f_j(\rho_j(t,0))=0.\nonumber 
\end{alignat*}
with $f_h^{\max}:=\max \{f_h(\rho)\,|\,\rho\in K_h\}$.
\end{definition}

In the following we will illustrate in detail why the optimization problem in Definition~\ref{def:HoldenRisebro} can be interpreted as a maximum entropy dissipation problem at the junction. 
The concept of maximum entropy dissipation was used in different situations in kinetic and hyperbolic theory before. Let us mention the work by Dafermos~\cite{Da1973} who applied an entropy rate admissibility criterion to select solutions to hyperbolic conservation laws. A maximum entropy dissipation problem at the junction is used in \cite{HHW2020} to select a kinetic coupling condition which converges to a macroscopic coupling for the system of isentropic gas dynamics. 
 
We introduce a slight modification of Definition~\ref{def:HoldenRisebro} involving entropy fluxes $\hat G_h$ corresponding to entropies $\hat\eta_h$ satisfying one the following assumptions:

\begin{enumerate}[label={\bf{($\pmb \eta$.\Alph*)}}]
	\item\label{eq:Assumptionsetah1} $\hat\eta_h\colon \mathbb R\to \mathbb R$ is convex;
	\item\label{eq:Assumptionsetah2} $\hat\eta_h\colon \mathbb R\to \mathbb R$ is convex and $\hat\eta_h'\,f_h'\le 0$ almost everywhere on $K_h$.
\end{enumerate} 
Remark that \ref{eq:Assumptionsetah2} characterizes the convex functions $\hat \eta_h$ such that there exists $\rho\in K_h $ which is a minimal point of $\hat \eta_h$ and a maximal point of $f_h$. In other words this point $\rho$ is of minimal entropy and maximal flux. Later this additional condition will be needed to prove, under certain conditions, that the following definitions for the generalized Riemann problem are equivalent to Definition~\ref{def:HoldenRisebro}.


\begin{definition}\label{def:MacroMaximumEntropyProb}
	Fix constant initial data $\rho_{0,h}\in K_h$ and entropy pairs $(\hat \eta_h,\hat G_h)$ satisfying \ref{eq:Assumptionsetah2}.
	A solution to the Riemann problem in the sense of the maximum entropy dissipation problem is given by the restrictions of  $(\rho_1,\dots,\rho_{n+m})$ to $x\in \Omega_h$,
	where $\rho_h\colon [0,\infty)_t\times \mathbb R_x\to K_h$ are the weak entropy solutions to the standard Riemann problems with initial data
	\begin{equation*}
		\rho_h(0,x)=
		\begin{cases}
			\rho_{0,h}\quad &\text{if }x\in\Omega_h,\\
			\tilde\rho_h\quad &\text{if }x\notin\Omega_h,
		\end{cases}
	\end{equation*}
solving
	\begin{alignat*}{3} 
		\underset{\tilde\rho_h\in K_h}{\max} & \quad & & \sum_{i=1}^n \hat G_{i}(\rho_i(t,0+))-\sum_{j=n+1}^{n+m} \hat G_{j}(\rho_j(t,0-))\\
		\text{s.t.}&&& \sum_{i=1}^n f_i(\rho_i(t,0)) -\sum_{j=n+1}^{n+m} f_j(\rho_j(t,0))=0.\nonumber \end{alignat*}
\end{definition}
In Section 2 we will prove that the optimization problem in Definition~\ref{def:MacroMaximumEntropyProb} admits a unique solution if $\hat \eta_h$ are strictly convex and the equivalence between Definition~\ref{def:HoldenRisebro} and Definition~\ref{def:MacroMaximumEntropyProb} for suitable $\hat g_h$ and $\hat G_h$.

We give a motivation for Definition~\ref{def:MacroMaximumEntropyProb}:
Integration by parts applied to the entropy formulation of \eqref{eq:LWRmodel} gives the quantity
\begin{equation}\label{eq:EntropyDecayJunctionMacro2}
	\sum_{i=1}^{n}\hat G_{i}(\rho_i(t,0-))- \sum_{j=n+1}^{n+m}\hat G_{j}(\rho_{j}(t,0+))
\end{equation}
in the boundary integral at the junction. More precisely, \eqref{eq:EntropyDecayJunctionMacro2} measures the entropy dissipation at the junction. It is different to the quantity 
\begin{equation}\label{eq:HelpDerivationRiemanSolver}
	\sum_{i=1}^{n}\hat G_{i}(\rho_i(t,0+))- \sum_{j=n+1}^{n+m}\hat G_{j}(\rho_{j}(t,0-))
\end{equation}
involved in the optimization problem in Definition~\ref{def:MacroMaximumEntropyProb}. More precisely, the directions from which the traces are taken are exchanged. The traces in \eqref{eq:EntropyDecayJunctionMacro2} are the traces which are still visible in the generalized Riemann problem after the functions are restricted to $x\in\Omega_h$. The traces in \eqref{eq:HelpDerivationRiemanSolver} will disappear after the restriction.
Both quantities coincide if and only if there is no stationary shock at $x=0$ in one of the standard Riemann problems. In Definition~\ref{def:MacroMaximumEntropyProb} we use the quantity \eqref{eq:HelpDerivationRiemanSolver} and make the assumption that the entropy dissipation due to this stationary shocks cannot be traced back to the dynamics in the junction.

To allow also entropies $\hat \eta_h$ which are not strictly convex, we introduce the priority functions $\Pi^M$. 
For $M\in\mathbb R$, we define $\Pi^M=(\Pi^M_1,\dots,\Pi^M_{n+m})$ such that 
\begin{enumerate}[label={\bf{($\pmb {\Pi}$.\arabic*)}}]
	\item\label{eq:DefPiM} 
	the functions $\Pi^M_h\colon [0,1]\to \{\rho\in K_h\,|\,\hat\eta_h'(\rho-)\le M \le \hat\eta_h'(\rho+)\};\; z\mapsto \Pi^M_h(z)$ are increasing and surjective for $h=1,\dots,n+m$ and $M\in\mathbb R$.
\end{enumerate} 

We give the following definition for the generalized Riemann problem. 

\begin{definition}\label{def:RiemannSolver}
	Fix constant initial data $\rho_{0,h}\in K_h$, entropy pairs $(\hat \eta_h,\hat G_h)$ and $\Pi^M$ satisfying \ref{eq:Assumptionsetah1} and \ref{eq:DefPiM}. A solution to the generalized Riemann problem is given by the restrictions of $\rho_h(t,x)$ to $x\in\Omega_h$, where $\rho_h$ are the solutions to the standard Riemann problem with initial data
	\begin{equation*}
		\rho_h(0,x)=
		\begin{cases}
			\rho_{0,h}&\text{if }x\in\Omega_h,\\
			\hat\rho_h &\text{if }x\notin\Omega_h.
		\end{cases}
	\end{equation*}
	The states $\hat\rho_h\in K_h$ are chosen such that
	\begin{gather*}
		\hat\eta_h'(\hat\rho_h-)\le M\le \hat\eta_h'(\hat\rho_h+)\quad\text{and}\quad 
		\hat\rho_h=\Pi^M_h(z)\quad\text{for }M\in\mathbb R,z\in [0,1],\\
		\sum_{i=1}^n f_i(\rho_i(t,0)) -\sum_{j=n+1}^{n+m} f_j(\rho_j(t,0))=0.
	\end{gather*}
	We define the Riemann solver by $\mathcal {RP}(\rho_{0})=(\rho_1(t,0-),\dots,\rho_n(t,0-),\rho_{n+1}(t,0+),\dots,\rho_{n+m}(t,0+))$ as a map $\mathcal {RP}\colon \bigtimes_{h=1}^{n+m}K_h\to \bigtimes_{h=1}^{n+m}K_h$.
\end{definition}

Solutions to the generalized Cauchy problem are defined as follows.
\begin{definition}\label{def:CauchyProblemMacro}
	Let $\rho_{0,h}\in L^1(\Omega_h,K_h)$ and let $\mathcal{RP}$ be as in Definition~\ref{def:RiemannSolver}. We call $(\rho_1,\dots,\rho_{n+m})$ with $\rho_h\colon (0,\infty)_t\times\Omega_h \to  K_h$ a solution to the generalized Cauchy problem if
	\begin{itemize}
		\item for any entropy pair $(\eta_h,G_h)$ and for any test function $\phi\in\mathcal D(\mathbb R_t\times\Omega_h)$, $\phi\ge 0$ there holds
		\begin{equation*}
			\iint_{\mathbb R_t\times\Omega_h}\eta_h(\rho_h)\,\partial_t\phi+G_h(\rho_h)\,\partial_x\phi\d t\d x+\int_{\Omega_h}\eta_h(\rho_{0,h}(x))\,\phi(0,x)\d x\ge 0,
		\end{equation*}
		\item there exist $\rho_h(t,0)\in L^\infty((0,\infty)_t,K_h)$ such that 
			\begin{alignat*}{2}
				&G_h(\rho_h(t,0))=\lim_{\substack{x\to 0\\ x\in\Omega_h}}G_h(\rho_h(t,x))&\qquad &\text{for a.e. }t>0,
			\end{alignat*}
		for all entropy fluxes $G_h$ and
			\begin{alignat*}{2}
				&\mathcal{RP}(\rho(t,0))=\rho(t,0)&\qquad &\text{for a.e. }t>0.
			\end{alignat*}
	\end{itemize}
\end{definition}

To prove existence of solutions in the sense of Definition~\ref{def:CauchyProblemMacro}, we use an approximation by a BGK model \cite{PeTa1991} 
\begin{equation*}
	\partial_t g_h(t,x,\xi)+f'_h(\xi)\,\partial_x g_h(t,x,\xi)=\frac{\chi(\rho_{g_h}(t,x),\xi)-g_h(t,x,\xi)}{\epsilon}\qquad t>0,x\in \Omega_h,\xi\in K_h,
\end{equation*}
with $g_h=g_h(t,x,\xi)$, $\rho_{g_h}(t,x)=\int_{ K_h}g_h(t,x,\xi)\d{\xi}$ and 
	$\chi(\rho,\xi)=\sgn(\xi)$ if $\xi\cdot(\rho-\xi)> 0$, $\chi(\rho,\xi)=0$ else.
At the junction a kinetic coupling condition $\Psi$ defined by a kinetic maximum entropy dissipation problem is imposed. Then, we proceed as follows: The convergence of the kinetic solutions in the interior of $\Omega_h$ is justified by a compactness result by Panov \cite{Pa1995}. The existence of stationary solutions to the BGK model and a careful analysis of the traces at the junction is needed to pass to the limit in the coupling condition. 
Uniqueness of the solutions to the Cauchy problem is obtained by an $L^1$--contraction property of the coupling condition in combination with classical Kru\v{z}kov--estimates. This approach is similar to the approach in \cite{ACD2017} for a vanishing viscosity approximation. Our main result is the following.

\begin{theorem}\label{thm:MainResult}
	Let $\rho_{0,h}\in L^1(\Omega_h,K_h)$, \ref{eq:Assumptionsfh1}--\ref{eq:Assumptionsfh4}, \ref{eq:Assumptionsetah1} and \ref{eq:DefPiM} hold. Let $g_h^\epsilon$ be the solution to the BGK model with coupling condition $\Psi$, initial condition $g_h^\epsilon(0,x,\xi)=\chi(\rho_{0,h}(x),\xi)$ and $\epsilon>0$ obtained in Theorem~\ref{thm:ExistenceBGKKineticModel}. Then, there exist $\rho_h\in C([0,\infty)_t,L^1(\Omega_h,K_h))$ such that, after possibly taking a subsequence, $\int_{\mathbb R}g_h^\epsilon\d\xi\to \rho_h$ as $\epsilon\to 0$ for a.e.~$(t,x)\in (0,\infty)\times\Omega_h$. $(\rho_1,\dots,\rho_{n+m})$ is the unique solution to the generalized Cauchy problem in the sense of Definition~\ref{def:CauchyProblemMacro}.
\end{theorem}

We give some remarks on the assumptions we made. \ref{eq:Assumptionsfh4} is only needed to pass to the limit in the kinetic approximation and \ref{eq:Assumptionsetah2} is needed for Definition~\ref{def:MacroMaximumEntropyProb}. Definition~\ref{def:RiemannSolver} can be studied without using the stronger assumption \ref{eq:Assumptionsetah2}.
\ref{eq:DefPiM} is only needed for entropies $\hat\eta_h$ which are not strictly convex. More precisely, the coupling condition in Definition~\ref{def:RiemannSolver} depends on the choice of $\Pi^M$ only if $\mathcal L(\{x\in \Gamma_h^\mathrm{out}\,|\,\hat\eta_h'(x-)\le M\le \hat\eta'_h(x+)\})\neq 0$ for two or more roads $h=1,\dots,n+m$.

The coupling conditions depend strongly on the choice of the entropies $(\hat\eta_h,\hat G_h)$. Therefore, they have to be chosen carefully depending on the desired modelling.   
We give two examples which we will study in more detail in Section 6:
\begin{enumerate}[label=(\roman*)]
	\item The choice
	\begin{equation}\label{eq:SameEntropyHatEta}
		\hat\eta_h=\hat\eta_1\qquad\text{for }h=1,\dots,n+m,
	\end{equation}
	leads to the coupling condition in \cite{ACD2017,CoGa2010} obtained by a vanishing viscosity approach.
	\item The choice
	\begin{equation}\label{eq:FluxEntropyHatEta}
		\hat\eta_h=-f_h \qquad\text{for }h=1,\dots,n+m,
	\end{equation}
	leads to a coupling condition which prioritizes the outgoing characteristics at the junction with large velocity $f_h'$. The coupling condition can be interpreted such that the drivers maximize their velocity at the junction.
\end{enumerate}

\subsection{Outline of the paper} The paper is organized as follows. In Section 2 we study Definitions~\ref{def:HoldenRisebro}--\ref{def:RiemannSolver} in detail and prove relations between them. In Section 3 we introduce the BGK model and the kinetic coupling condition $\Psi$. In Section 4 we prove well-posedness of the kinetic model. In Section~5 we rigorously justify the limit of the kinetic model and obtain existence of solutions to the LWR model. Furthermore, we prove uniqueness and an $L^1$--contraction property of the solutions. In Section 6 we study an explicit example for the Riemann problem and different choices of $\hat\eta_h$. In Section 7 we give an outlook for future research.

\section{The LWR model on Networks}
\subsection{Introductory remarks}
\subsubsection*{Entropy solutions}

Entropy pairs for the LWR model are given by 
\begin{equation}\label{eq:DefEntropyPairs}
	(\eta_h,G_{h})\colon\mathbb R\to \mathbb R^2\;\text{ such that }\;\eta_h\;\text{ is convex and }\; G'_{h}=\eta'_h\, f_h'.
\end{equation}
Every convex function $\eta_h\colon\mathbb R\to \mathbb R$ is an entropy with entropy flux
\begin{equation}
	G_{h}(\rho)=\int_{0}^\rho \eta'_h(y)\, f'_h(y)\d{y}+G_{h}(0)\qquad \text{for } \rho\in \mathbb R.
\end{equation}
Throughout this paper we consider weak entropy solutions to \eqref{eq:LWRmodel} in the sense that 
\begin{equation}\label{eq:WeakEntropySolution}
	\partial_t \eta_h(\rho_h)+\partial_x G_{h}(\rho_h)\le 0\qquad t>0,x\in\Omega_h,
\end{equation}
holds in the distributional sense for all entropy pairs $(\eta_h,G_{h})$. 
Let us recall the definition of some special entropies:
\begin{itemize}
	\item The relative entropy and relative entropy flux of $\rho\in \mathbb R$ w.r.t. $k\in \mathbb R$ and $(\eta_h,G_{h})$ are given by
	\begin{align}
	\begin{split}
		\eta_h(\rho \,|\, k)&=\eta_h(\rho)-\eta_h(k)-\eta'_h(k)\left(\rho-k\right),\\
		G_{h}(\rho\, | \, k)&=G_{h}(\rho)-G_{h}(k)-\eta'_h(k)\left(f_h(\rho)-f_h(k)\right).
	\end{split}
	\end{align}
	\item The Kru\v{z}kov entropy pair of $\rho\in\mathbb R$ w.r.t.~$k\in\mathbb R$ is given by
	\begin{equation}\label{eq:DefKruzkovEntropyPair}
		\eta_h(\rho,k)=|\rho-k|,\quad G_h(\rho,k)=\sgn(\rho-k)\left(f_h(\rho)-f_h(k)\right).
	\end{equation}
\end{itemize}

\subsubsection*{Some Notations}
We define the sets of states $\rho\in K_h$ with incoming characteristics at the junction 
\begin{align}\label{eq:HelpGammaIn}
	\Gamma_h^\mathrm{in} =
	\begin{cases}
		\operatorname{Int}\,\{\rho\in K_h\,|\, f'_h(\rho)\le 0\} & h=1,\dots,n,\\
		\operatorname{Int}\,\{\rho\in K_h\,|\, f'_h(\rho)\ge  0\} & h=n+1,\dots,n+m,
	\end{cases}
\end{align}
and with outgoing characteristics 
\begin{align}\label{eq:HelpGammaOut}
	\Gamma_h^\mathrm{out} =
	\begin{cases}
		\overline{\{\rho\in K_h\,|\, f'_h(\rho)> 0\}}& h=1,\dots,n,\\
		\overline{\{\rho\in K_h\,|\, f'_h(\rho) <0\}}& h=n+1,\dots,n+m,
	\end{cases}
\end{align}
where $\operatorname{Int}$ denotes the interior and the overline $\overline{\phantom{x}}$ the closure w.r.t.~the topological space $K_h$.

To handle the incoming and outgoing roads simultaneously and to simplify notation, we define
	\begin{equation}
		\pm_h=\begin{cases}
			 +\quad & h=1,\dots,n,\\
			 - \quad & h=n+1,\dots,n+m.
		\end{cases}	
	\end{equation}
	for $h=1,\dots,n+m$.
	This notation is used to handle different signs in the case of incoming/outgoing roads.

Furthermore, we consider only the representatives of $ \eta'_h,f'_h\in L^\infty(K_h)$ which are monotone in the classical sense. This allows to evaluate $\eta_h'$ and $f_h'$ point-wise. Our results hold true for any monotone representative of $\eta'_h,f_h'$. Furthermore, we will use the convention
\begin{equation}\label{eq:HelpConventrionDerHatEtaBoundKh}
	\eta'_h(a_h-)=-\infty\quad\text{and}\quad \eta'_h(b_h+)=\infty.
\end{equation}

\subsubsection*{Strong traces of entropy solutions}
Solutions to scalar conservation laws do not have strong traces in general. We state the following theorem which is obtained by combining the results from \cite{KwVa2007,Pa2006}. It gives sense to the initial trace $t=0$ and to the left-- and right--traces at the junction $x=0$. 
\begin{theorem}\label{thm:HelpExistenceTrace} 
	Let $\rho_h\in L^\infty((0,\infty)_t \times \mathbb R_x,K_h)$ be a weak entropy solution to the LWR model with flux $f_h\in C^{0,1}(K_h)$ for some $h=1,\dots,n+m$. Then, there exists
	\begin{enumerate}[label=(\roman*)]
		\item the initial trace 
	\begin{equation*}
		\rho_h(0,x)\in L^\infty(\mathbb R_x,K_h)
	\end{equation*}
	such that for all entropies $\eta_h$:
	\begin{equation*}
		\eta_h(\rho(s,x))\to \eta_h(\rho_h(0,x)) \qquad\text{ in }L^1_\mathrm{loc}(\mathbb R_x)\text{ as }s\to 0 ,s>0 ;
	\end{equation*}
	\item the boundary trace
	\begin{equation*}
		L^\infty((0,\infty)_t,K_h)\ni \rho_h(t,0)=
		\begin{cases}
			\rho_h(t,0-)\quad &h=1,\dots,n,\\
			 \rho_h(t,0+)\quad &h=n+1,\dots,n+m,
		\end{cases}
	\end{equation*}
	at the junction such that for all entropy fluxes $G_h$:
	\begin{equation*}
		G_h(\rho(t,s))\to G_h(\rho_h(t,0)) \qquad \text{ in }L^1_\mathrm{loc}((0,\infty)_t)\text{ as }s\to 0 ,s\in \Omega_h;
	\end{equation*}
	\item the trace
	\begin{equation*}
		 L^\infty((0,\infty)_t,K_h)\ni\rho_h(t,0\pm_h)=
		\begin{cases}
			\rho_h(t,0+)\quad &h=1,\dots,n,\\
			 \rho_h(t,0-)\quad &h=n+1,\dots,n+m,
		\end{cases}
	\end{equation*} 
	at the junction such that for all entropy fluxes $G_h$:
	\begin{equation*}
		G_h(\rho(t,s))\to G_h(\rho_h(t,\pm_h 0)) \qquad \text{ in }L^1_\mathrm{loc}((0,\infty)_t)\text{ as }s\to 0 ,s\notin \Omega_h.
	\end{equation*}
	\end{enumerate}

\end{theorem} 
Later, the values $\rho_h(t,0)$ will satisfy the coupling condition in Definition~\ref{def:CauchyProblemMacro}. The traces $\rho_h(t,0\pm_h)$ appear in the construction using standard Riemann problems, e.g.\ in Definition~\ref{def:MacroMaximumEntropyProb}. The traces $\rho_h(t,0\pm_h)$ will dissapear, after restricting the solutions to $\Omega_h$.

The above traces are attained in the strong $L^1_{\mathrm{loc}}$--sense. This allows us to apply a divergence theorem using the theory of divergence measure fields \cite{An1983}.

Applying the divergence theorem to the weak formulation of the LWR model and using a sequence of test functions concentrating at $x=0$, we obtain the Rankine--Hugoniot--conditions 
\begin{equation*}
	G_h(\rho_h(t,0\pm_h))-G_h(\rho_h(t,0))
	\begin{cases}
		\ge 0 \quad &h=1,\dots,n,\\
		\le 0\quad &h=n+1,\dots,n+m,
	\end{cases}
\end{equation*}
for all entropy fluxes $G_h$. 
Therefore, $\rho_h(t,0)$ and $\rho_h(t,0\pm_h)$ coincide or are connected by a stationary shock. We conclude 
\begin{equation}
	f_h( \rho_h(t,0))=f_h( \rho_h(t,0+))=f_h( \rho_h(t,0-)).
\end{equation}

\subsection{The LWR model at the junction}

Having the physical application in mind, it is reasonable to assume that the total mass (the total number of cars) in the network is conserved in time
\begin{equation}\label{eq:MassConservationJunctionHelp}
	\sum_{h=1}^{n+m}\int_{\Omega_h}\rho_h(T,x)\d{x}=\sum_{h=1}^{n+m}\int_{\Omega_h}\rho_{0,h}(x)\d{x} \qquad \text{for every }T> 0,
\end{equation}
for $\rho_{0,h}\in L^1(\Omega_h,K_h)$.
Integration by parts applied to \eqref{eq:LWRmodel} together with \ref{eq:Assumptionsfh3} and \eqref{eq:MassConservationJunctionHelp} leads to the Rankine--Hugoniot--type condition
\begin{equation}\label{eq:MassConservationJunction}
	\sum_{i=1}^n f_i(\rho_i(t,0))=\sum_{j=n+1}^{n+m}f_j(\rho_j(t,0))\qquad \text{for a.e. }t>0.
\end{equation}
Notice that there is another common condition in the literature based on the priorities of the drivers at the junction. This condition implies \eqref{eq:MassConservationJunction}. More details and a short outlook how our techniques could be applied to this condition is given in Section~7.1. 

We continue with the entropy dissipation at the junction. For the LWR model on the real line, the total entropy decreases, i.e. 
\begin{equation*}
	\int_{\mathbb R}\eta(\rho(T_2,x))\d x\le \int_{\mathbb R} \eta(\rho(T_1,x))\d{x}\quad\text{for every }T_2\ge T_1\ge 0,
\end{equation*}
for every entropy $\eta$ with $\eta(0)=0$ and initial data $\rho_0\in L^1(\mathbb R,K_h)$. We aim to extend this property to networks. Therefore, we choose suitable entropies $(\eta_1,\dots,\eta_{n+m})$ with $\eta_h(0)=0$
and assume 
\begin{equation}\label{eq:TotalEntropyDecayNetworkMacro}
	\sum_{h=1}^{n+m}\int_{\Omega_h}\eta_h(\rho_h(T_2,x))\d x\le \sum_{h=1}^{n+m}\int_{\Omega_h}\eta_h(\rho_{h}(T_1,x))\d{x}\quad\text{for every }T_2\ge T_1\ge 0,
\end{equation}
and for $\rho_{0,h}\in L^1(\Omega_h,K_h)$. Analogous to \ref{eq:Assumptionsfh3} we make the assumption
\begin{equation}\label{eq:constraintEntropiesintheEnergyDecayOnTheNetwork}
	\sum_{i=1}^{n}G_i(0)-\sum_{j=n+1}^{n+m}G_j(0)=0.
\end{equation}
Integration by parts applied to \eqref{eq:WeakEntropySolution} together with \eqref{eq:TotalEntropyDecayNetworkMacro} and \eqref{eq:constraintEntropiesintheEnergyDecayOnTheNetwork} leads to 
\begin{equation}\label{eq:EntropyDecayJunctionMacro}
	\sum_{i=1}^{n}G_{i}(\rho_i(t,0))- \sum_{j=n+1}^{n+m}G_{j}(\rho_{j}(t,0))\ge 0\qquad\text{for a.e. }t> 0.
\end{equation}
We cannot expect that \eqref{eq:EntropyDecayJunctionMacro} holds for every choice of $\eta_h$. Otherwise, we get $\rho_h(t,0)=0$ (set $G_h(\rho)=f_h(\rho)-f_h(0)$ and $G_h(\rho)=f_h(0)-f_h(\rho)$ with $G_{h'}\equiv 0$ for $h'\neq h$). The most we can expect is that \eqref{eq:EntropyDecayJunctionMacro} holds for a suitable class of entropies $(\eta_1,\dots,\eta_{n+m})$.

\subsection{The Riemann solver}

We prove existence and uniqueness of solutions to the generalized Riemann problem in Definition~\ref{def:RiemannSolver} and an $L^1$-contraction property of the corresponding Riemann solver $\mathcal{RP}$.

\begin{proposition}\label{Prop:ExistenceRiemannProblem}
	Let $\rho_{0,h}\in K_h$ be constant and let \ref{eq:Assumptionsfh1}--\ref{eq:Assumptionsfh3}, \ref{eq:Assumptionsetah1} and \ref{eq:DefPiM} hold.
	There exists a unique solution to the generalized Riemann problem in the sense of Definition~\ref{def:RiemannSolver}. 
	The solution is stationary if and only if $\mathcal{RP}(\rho_0)=\rho_0$.
\end{proposition}
\begin{proof}
	The proof uses similar arguments as in \cite[Lemma~2.3]{ACD2017}.
	For fixed $M\in\mathbb R$ and $z\in [0,1]$, the standard Riemann problems with initial data 
	\begin{equation*}
		\rho_h(0,x)=
		\begin{cases}
			\rho_{0,h}&\text{if }x\in\Omega_h,\\
			\hat\rho_h &\text{if }x\notin\Omega_h,
		\end{cases}
	\end{equation*}
	and $\hat\rho_h=\Pi^{M}_h(z)$ can be solved uniquely. It remains to find $M\in\mathbb R$ and $z\in [0,1]$ such that the conservation of mass holds. 
	
	The functions $M\mapsto \Pi^{M}_h(0)$ are increasing and of bounded variation since $\hat\eta_h$ are convex and by definition of $\Pi^{M}_h$. We have 
	\begin{align*}
		\lim_{M\to -\infty}\Pi^{M}_h(0)=a_h\quad\text{and}\quad 
		\lim_{M\to +\infty}\Pi^{M}_h(0)=b_h.
	\end{align*}
	Furthermore, $\hat\rho_h\mapsto f_h(\rho_h(t,0))$ is continuous and monotone since  
	\begin{alignat}{2}\label{eq:HelpGodunovFluxJunction1}
		f_i(\rho_i(t,0))&=
		\begin{cases}\displaystyle
			\min_{s\in [\rho_{0,i},\hat\rho_i]}f_i(s)\quad &\text{if } \hat\rho_i\ge \rho_{0,i}\\
			\displaystyle
			\max_{s\in [\hat\rho_i,\rho_{0,i}]}f_i(s)\quad &\text{if } \hat\rho_i\le \rho_{0,i}
		\end{cases}\qquad &&\text{for }i=1,\dots,n,\\
		f_j(\rho_j(t,0))&=
		\begin{cases}\displaystyle
			\min_{s\in [\hat\rho_j,\rho_{0,j},]}f_j(s)\quad &\text{if } \rho_{0,j}\ge \hat\rho_j\\
			\displaystyle
			\max_{s\in [\rho_{0,j},\hat\rho_j]}f_j(s)\quad &\text{if } \rho_{0,j}\le \hat\rho_j
		\end{cases}\qquad &&\text{for }j=n+1,\dots,n+m.\label{eq:HelpGodunovFluxJunction2}
	\end{alignat}
	The above formulae follow from the structure of the standard Riemann problem and from the concavity of $f_h$.
	
	An intermediate value argument applied to the following decreasing function of bounded variation
	\begin{equation*}
		\Xi(M)= \sum_{i=1}^n f_i(\rho_i(t,0)) -\sum_{j=n+1}^{n+m} f_j(\rho_j(t,0))\qquad\text{for } \hat\rho_h=\Pi^{M}_h(0)
	\end{equation*} 
	in combination with 
	\begin{align*}
		\Xi(M)
		\begin{cases}
			\ge 0 \qquad\text{for }M\to -\infty,\\
			\le 0 \qquad\text{for }M\to +\infty,
		\end{cases}
	\end{align*}
	leads to two possible cases: 
	In the first case there exists an $M\in\mathbb R$ such that $\Xi(M)=0$. 
	Otherwise, there exists an $M\in\mathbb R$ such that 
		\begin{align*}
				\Xi(M-)\ge 0 \quad\text{and}\quad
				\Xi(M+)\le 0 .
		\end{align*}
	Due to \ref{eq:DefPiM}, there exists $z\in[0,1]$ such that 
		\begin{align*}
			\sum_{i=1}^n f_i(\rho_i(t,0)) -\sum_{j=n+1}^{n+m} f_j(\rho_j(t,0))=0 \qquad \text{for }
			\hat\rho_h=\Pi^{M}_h(z).
		\end{align*}
	The existence of a solution follows. 
	
	By monotonicity of $\hat\rho_h\mapsto f_h(\rho_h(t,0))$, the uniqueness of $f_h(\rho_h(t,0))$ is obtained. The uniqueness of $\rho_h\colon (0,\infty)_t\times\Omega_h\to K_h$ follows from the uniqueness of $f_h(\rho_h(t,0))$ and the structure of the standard Riemann problem. By the uniqueness we obtain $\rho_h(t,x)=\rho_{0,h}$ for a.e.~$t,x\in (0,\infty)_t\times\Omega_h$ if $\mathcal{RP}(\rho_0)=\rho_0$.
\end{proof}

We continue with the $L^1$--contraction property.
\begin{proposition}\label{prop:ContractionPropertyJunction}
	Let \ref{eq:Assumptionsfh1}--\ref{eq:Assumptionsfh3}, \ref{eq:Assumptionsetah1} and \ref{eq:DefPiM} hold. Let $\mathcal{RP}(\rho^s)=\rho^s$ hold for $\rho_{h}^s\in K_h$, $s=1,2$. Then, we have
	\begin{equation*}
		\sum_{i=1}^n G_i(\rho_i^1,\rho_i^2)-\sum_{j=n+1}^{n+m}G_j(\rho_j^1,\rho_j^2)\ge 0,
	\end{equation*}
	for $G_h(\rho_h^1,\rho_h^2)$ being the Kru\v{z}kov entropy flux defined in \eqref{eq:DefKruzkovEntropyPair}.
\end{proposition}
\begin{proof}
Let $M^s, z^s$ be as in Definition~\ref{def:RiemannSolver} with the constant initial data $\rho_{0,h}=\rho_h^s$. Without loss of generality we can assume that $\Pi_1^{M^1}(z^1)-\Pi_1^{M^2}(z^2)\neq 0$ if $\Pi_h^{M^1}(z^1)-\Pi_h^{M^2}(z^2)\neq 0$ for some $h=1,\dots,n+m$. If this is not the case, we can reorder the indices and possibly exchange the role of $i$ and $j$. We define 
	\begin{equation*}
		\theta :=\sgn(\Pi_1^{M^1}(z^1)-\Pi_1^{M^2}(z^2)).
	\end{equation*}
	By monotonicity of $\Pi_h^M(z)$ w.r.t.~$M$ and $z$, we get
	\begin{equation*}
		\sgn(\Pi_h^{M^1}(z^1)-\Pi_h^{M^2}(z^2))=\theta \qquad \text{for }h=1,\dots,n+m,
	\end{equation*}
	with the convention $\sgn(0)=\theta$. \\ 
\noindent
\underline{Step 1:} We compute
	\begin{align*}
		&\sum_{i=1}^n G_i(\rho_i^1,\rho_i^2)-\sum_{j=n+1}^{n+m}G_j(\rho_j^1,\rho_j^2)\\
		&= \sum_{i=1}^n\sgn(\rho_i^1-\rho_i^2)( f_i(\rho_i^1)-f_i(\rho_i^2))-\sum_{j=n+1}^{n+m}\sgn(\rho_j^1-\rho_j^2)(f_j(\rho_j^1)-f_j(\rho_j^2))\\
		&\quad -\theta\left( \sum_{i=1}^n( f_i(\rho_i^1)-f_i(\rho_i^2))-\sum_{j=n+1}^{n+m}(f_j(\rho_j^1)-f_j(\rho_j^2))\right)\\
		&=\sum_{i=1}^n(\sgn(\rho_i^1-\rho_i^2)-\theta)( f_i(\rho_i^1)-f_i(\rho_i^2)) -\sum_{j=n+1}^{n+m}(\sgn(\rho_j^1-\rho_j^2)-\theta)(f_j(\rho_j^1)-f_j(\rho_j^2))\\
		&\ge\sum_{i\in\mathcal A}(\sgn(\rho_i^1-\rho_i^2)-\theta)( f_i(\rho_i^1)-f_i(\rho_i^2)) -\sum_{j\in\mathcal B}(\sgn(\rho_j^1-\rho_j^2)-\theta)(f_j(\rho_j^1)-f_j(\rho_j^2)).
	\end{align*}
The sets $\mathcal A\subset\{1,\dots,n\}$ and $\mathcal B\subset\{1,\dots,n\}$ are defined by 
	\begin{align*}
		\mathcal A&=\{i\,|\, (\sgn(\rho_i^1-\rho_i^2)-\theta)( f_i(\rho_i^1)-f_i(\rho_i^2))<0\},\\
		\mathcal B&=\{j\, | \, (\sgn(\rho_j^1-\rho_j^2)-\theta)( f_j(\rho_j^1)-f_j(\rho_j^2))>0\}.
	\end{align*}
Since $\rho_i^1\neq \rho_i^2$ for $i\in\mathcal A$ (otherwise $f_i(\rho_i^1)=f_i(\rho_i^2)$), we get 
	\begin{align*}
	\mathcal A&=\{i\,|\,\sgn(\rho_i^1-\rho_i^2)\neq\theta\text{ and }\sgn(\rho_i^1-\rho_i^2)\,( f_i(\rho_i^1)-f_i(\rho_i^2))< 0\}.	
	\end{align*}
	
\noindent
\underline{Step 2:} We collect some properties of $\mathcal A$:
\begin{itemize}
	\item $\rho_i^1\neq \rho_i^2$ as already observed above;
	\item {$\{\rho_i^1,\rho_i^2\}\cap \Gamma_i^\mathrm{out}\neq \{\rho_i^1,\rho_i^2\}$} since otherwise $\sgn(\rho_i^1-\rho_i^2)=\sgn(\Pi_i^{M^1}(z^1)-\Pi_i^{M^2}(z^2))=\theta$;
	\item {$\{\rho_i^1,\rho_i^2\}\cap \Gamma_i^\mathrm{in}\neq \{\rho_i^1,\rho_i^2\}$} since otherwise $\sgn(\rho_i^1-\rho_i^2)\,( f_i(\rho_i^1)-f_i(\rho_i^2))\ge 0$ due to the fact that $f_i$ is increasing on $\Gamma_i^\mathrm{in}$. 
\end{itemize}
We conclude that exactly one element of $\{\rho_i^1,\rho_i^2\}$ lies in $\Gamma_i^\mathrm{out}$ and exactly one lies in $\Gamma_i^\mathrm{in}$. 

\noindent
\underline{Step 3:} 
Without loss of generality we assume that $\rho_i^1\in \Gamma_i^\mathrm{out}$ and $\rho_i^2\in \Gamma_i^\mathrm{in}$. Then, there exists $ R_i^2\in K_i$ such that $ R_i^2=\Pi_i^{M^2}(z^2)$. Since $1=\sgn(\rho_i^1-\rho_i^2)\neq\theta=\sgn(\Pi_i^{M^1}(z^1)-\Pi_i^{M^2}(z^2))=\sgn(\rho_i^1-R_i^2)$, we have $\rho_i^1\le R_i^2\in \Gamma_i^\mathrm{out}$. We obtain
\begin{itemize}
	\item $f_i(R_i^2)\le f_i(\rho_i^1)$ by monotonicity of $f_i$ on $\Gamma_i^\mathrm{out}$ and $\rho_i^1\le R_i^2$;
	\item $f_i(\rho_i^{1})<f_i(\rho_i^{2})$ since $\sgn(\rho_i^1-\rho_i^2)\,( f_i(\rho_i^1)-f_i(\rho_i^2))< 0$ and $\sgn(\rho_i^1-\rho_i^2)>0$;
	\item $f_i(\rho_i^{2})\le f_i(R_i^2)$ by the construction of $\rho_i^{2}$ in Definition~\ref{def:RiemannSolver} based on standard Riemann problems with right--hand state $R_i^2=\Pi_i^{M^2}(z^2)\in \Gamma_i^\mathrm{out}$. 
\end{itemize}
The three estimates lead to the contradiction 
	\begin{equation*}
		f_i(R_i^2)\le f_i(\rho_i^1)<f_i(\rho_i^{2})\le f_i(R_i^2)
	\end{equation*}
and we conclude that $\mathcal A=\emptyset$. With similar arguments $\mathcal B=\emptyset$ follows and we conclude 
\begin{align*}
		&\sum_{i=1}^n G_i(\rho_i^1,\rho_i^2)-\sum_{j=n+1}^{n+m}G_j(\rho_j^1,\rho_j^2)\ge 0.
\end{align*}
\end{proof}
Later we will see that the inequalities on the Kru\v{z}kov entropy fluxes give the natural conditions leading to the $L^1$--contraction property for the Cauchy problem.

\subsection{The maximum entropy dissipation problem}

\begin{proposition}\label{prop:RelationRiemannProblemMaximumEntropy}
	Let \ref{eq:Assumptionsfh1}--\ref{eq:Assumptionsfh3}, \ref{eq:Assumptionsetah2} and \ref{eq:DefPiM} hold. There exists a solution to the maximum entropy dissipation problem in Definition~\ref{def:MacroMaximumEntropyProb} solving the coupling condition 
	\begin{equation*}
		\rho_h(t,0)=\mathcal{RP}_h(\rho_{0}).
	\end{equation*}
	The solution is unique if $\hat\eta_h$ are strictly convex.
\end{proposition}
\begin{proof}~\\
	\noindent
	\textit{Existence:} Replacing $\hat\rho_h$ by $\tilde\rho_h$ in \eqref{eq:HelpGodunovFluxJunction1}--\eqref{eq:HelpGodunovFluxJunction2} implies that $f_h(\rho_h(t,0))$ is a continuous function of $\tilde\rho_h$.
	Furthermore, $\rho_i(t,0)$ are increasing and upper semi-continuous w.r.t. $\tilde\rho_i$. The continuous entropy fluxes $\hat G_h$ are decreasing due to \ref{eq:Assumptionsetah2}.
	We conclude that $\hat G_i(\rho_i(t,0))$ is decreasing and lower semi-continuous w.r.t. $\tilde\rho_i$. By the same arguments we obtain that $\hat G_j(\rho_j(t,0))$ is decreasing and upper semi-continuous w.r.t. $\tilde\rho_j$.
	We obtain the existence by compactness of $K_h$.
	
	\noindent
	\textit{Reformulation of Definition~\ref{def:MacroMaximumEntropyProb}:}
	By using \eqref{eq:HelpGodunovFluxJunction1}--\eqref{eq:HelpGodunovFluxJunction2} again, it can be easily checked that  
	\begin{align}\label{eq:HelEquivalenceRiemannSolverOptimziationProb1}
	f_h(\rho_h(t,0)) 
	\text{ is constant for }
		\begin{cases}
			\tilde\rho_h\in \Gamma_h^\mathrm{in}\quad &\text{ if }\rho_{0,h}\in\Gamma_h^\mathrm{out},\\
			\tilde\rho_h\notin \{\rho\in\Gamma_h^\mathrm{out}\,|\,f_h(\rho)\le f_h(\rho_{0,h})\}\quad &\text{ if }\rho_{0,h}\in\Gamma_h^\mathrm{in}.
		\end{cases}
	\end{align}
	By \eqref{eq:HelEquivalenceRiemannSolverOptimziationProb1} and by the monotonicity of $\hat G_h(\rho_h(t,0\pm_h))$ w.r.t.~$\tilde\rho_h$, the following additional constraint can be added without changing the solution on $\Omega_h$:
	\begin{align}\label{eq:HelEquivalenceRiemannSolverOptimziationProb2}
	\begin{split}
		\tilde\rho_h\in \Gamma_h^\mathrm{out}\quad &\text{ if }\rho_{0,h}\in\Gamma_h^\mathrm{out},\\
		\tilde\rho_h\in\{\rho\in\Gamma_h^\mathrm{out}\,|\,f_h(\rho)\le f_h(\rho_{0,h})\}\quad &\text{ if }\rho_{0,h}\in\Gamma_h^\mathrm{in}.
	\end{split}
	\end{align}
	
	Remark that the values $\tilde\rho_h$ with $f_h(\tilde\rho_h)=f_h^{\max}=\max f_h(\Gamma_h^\mathrm{out})$ were not uniquely defined through the original problem. By adding the additional constraint \eqref{eq:HelEquivalenceRiemannSolverOptimziationProb2}, only the value $\tilde\rho_h=\operatorname{argmax}f_h(\Gamma_h^\mathrm{out})$ is admissible. 
	
	For $\tilde\rho_h$ satisfying \eqref{eq:HelEquivalenceRiemannSolverOptimziationProb2}, we have $\rho_h(t,0\pm_h)=\tilde\rho_h$, the objective function reads 
	\begin{align}\label{eq:HelEquivalenceRiemannSolverOptimziationProb22}
		&\sum_{i=1}^n \hat G_i(\tilde\rho_i) -\sum_{j=n+1}^{n+m} \hat G_j(\tilde\rho_j),
	\end{align}
	and the constraint in Definition~\ref{def:MacroMaximumEntropyProb} is
	\begin{align}\label{eq:HelEquivalenceRiemannSolverOptimziationProb23}
	 	&\sum_{i=1}^n f_i(\tilde\rho_i) -\sum_{j=n+1}^{n+m} f_j(\tilde\rho_j)=0.
	\end{align}
	We write the optimization problem in terms of $\tilde  F_h=f_h(\tilde\rho_h)$ similar to \cite{GHP2016}. This is equivalent to the original problem since $f_h$ is invertible on $\Gamma_h^\mathrm{out}$ and $\tilde\rho_h\in \Gamma_h^\mathrm{out}$ due to \eqref{eq:HelEquivalenceRiemannSolverOptimziationProb2}
. We obtain
	\begin{alignat}{3} \label{eq:HelEquivalenceRiemannSolverOptimziationProb3}
	\max & \quad & & \sum_{i=1}^n \hat G_{i}(f_i^{-1}|_{\Gamma_i^\mathrm{out}}(\tilde  F_i))-\sum_{j=n+1}^{n+m} \hat G_{j}(f_j^{-1}|_{\Gamma_j^\mathrm{out}}(\tilde  F_j))\\
	\text{s.t.}&&&  \sum_{i=1}^n  \tilde  F_i -\sum_{j=n+1}^{n+m} \tilde  F_j=0\nonumber \\
	&&& 0\le \tilde  F_h\le  f_h^{\max}\qquad\ \text{ if }\rho_{0,h}\in\Gamma_h^\mathrm{out}\nonumber\\
	&&&	0\le \tilde  F_h\le f_h(\rho_{0,h})\quad \text{ if }\rho_{0,h}\in\Gamma_h^\mathrm{in}\nonumber
	\end{alignat}

	\noindent
	\textit{The optimization problem \eqref{eq:HelEquivalenceRiemannSolverOptimziationProb3} has a concave objective function and affine constraints:}
	The function $\hat G_{i}(f_i^{-1}|_{\Gamma_i^\mathrm{out}}(\tilde  F_i))$ is concave:
		$f_i^{-1}|_{\Gamma_i^\mathrm{out}}$ is convex since $f_i|_{\Gamma_i^\mathrm{out}}$ is decreasing and concave.
		$\hat G_i$ is concave and decreasing on $\Gamma_i^\mathrm{out}$ since $\hat G'_i=\hat\eta_i'\,f_i'$ is decreasing and negative \ref{eq:Assumptionsetah2} on $\Gamma_i^\mathrm{out}$.
	The concavity of $-\hat G_{j}(f_j^{-1}|_{\Gamma_j^\mathrm{out}}(\tilde  F_j))$ follows by similar arguments.

	\noindent
	\textit{Application of the KKT-conditions to \eqref{eq:HelEquivalenceRiemannSolverOptimziationProb3}:} Due to the structure of  \eqref{eq:HelEquivalenceRiemannSolverOptimziationProb3}, the KKT-conditions are necessary and sufficient. They read:
	\begin{align*}
		&\pm_h \hat\eta_h'(f^{-1}_h|_{\Gamma_h^\mathrm{out}}(\tilde F_h)\mp_h)\le\pm_h \lambda+\mu_h -\nu_h\le \pm_h \hat\eta_h'(f^{-1}_h|_{\Gamma_h^\mathrm{out}}(\tilde F_h)\pm_h) \quad\text{for }h=1,\dots,n+m,\\
		&\mu_h=  0\quad\text{or}\quad \tilde F_h=\begin{cases}
			f_h^{\max}\quad &\text{ if }\rho_{0,h}\in\Gamma_h^\mathrm{out},\\
			f_h(\rho_{0,h})\quad &\text{ if }\rho_{0,h}\in\Gamma_h^\mathrm{in},\\
		\end{cases}\\
		&\nu_h=  0\quad\text{or}\quad \tilde F_h=0,
	\end{align*}
	for $\lambda\in\mathbb R, \mu_h,\nu_h\in [0,\infty)$.
	With the convention \eqref{eq:HelpConventrionDerHatEtaBoundKh},
	then there exists $M\in\mathbb R$ such that for every $h=1,\dots,n+m$ one of the following conditions hold
	\begin{align} 
	\begin{split}
		&\hat\eta_h'(f_h^{-1}|_{\Gamma_h^\mathrm{out}}(\tilde  F_h)-)\le M\le \hat\eta_h'(f_h^{-1}|_{\Gamma_h^\mathrm{out}}(\tilde  F_h)+),\\
		&\tilde  F_i=
			\begin{cases}
				f_i^{\operatorname{max}} &\text{if }\rho_{0,i}\in\Gamma_i^\mathrm{out}\\
				f_i(\rho_{0,i}) &\text{if }\rho_{0,i}\in\Gamma_i^\mathrm{in}
			\end{cases}
			\quad\text{with}\quad M\le \hat\eta_i'(f_i^{-1}|_{\Gamma_i^\mathrm{out}}(\tilde  F_i)-),\\
		&\tilde  F_j=
			\begin{cases}
				f_j^{\operatorname{max}} &\text{if }\rho_{0,j}\in\Gamma_j^\mathrm{out}\\
				f_j(\rho_{0,j}) &\text{if }\rho_{0,j}\in\Gamma_j^\mathrm{in}
			\end{cases}
			\quad\text{with}\quad \hat\eta_j'(f_j^{-1}|_{\Gamma_j^\mathrm{out}}(\tilde  F_j)+)\le M.
	\end{split}
	\end{align}
	
	\noindent
	\textit{Solutions to Definition~\ref{def:MacroMaximumEntropyProb}:}
	Using the fact that problem \eqref{eq:HelEquivalenceRiemannSolverOptimziationProb3} and Definition~\ref{def:MacroMaximumEntropyProb} are equivalent for $\tilde  F_h=f_h(\tilde\rho_h)$, we obtain that $\tilde \rho_h$ solves one of the following conditions
	\begin{align}\label{eq:HelEquivalenceRiemannSolverOptimziationProb4}
	\begin{split}
		&\hat\eta_h'(\tilde \rho_h-)\le M\le\hat\eta_h'(\tilde \rho_h+)\quad\text{for }h=1,\dots,n+m,\\
		&\tilde \rho_i=
		\begin{cases}
			\operatorname{min}\, \Gamma_i^\mathrm{out}\ &\text{if }\rho_{0,i}\in\Gamma_i^\mathrm{out}\\
			\operatorname{min}\, \{\Gamma_i^\mathrm{out}\,|\,f_i(\cdot)\le f_i(\rho_{0,i})\}\ &\text{if }\rho_{0,i}\in\Gamma_i^\mathrm{in}
		\end{cases}
		\quad\text{with }\hat\eta_i'(\tilde\rho_i-)\le M,\\
		&\tilde \rho_j=
		\begin{cases}
			\operatorname{max}\, \Gamma_j^\mathrm{out}\ &\text{if }\rho_{0,j}\in\Gamma_j^\mathrm{out}\\
			\operatorname{max}\, \{\Gamma_j^\mathrm{out}\,|\,f_j(\cdot)\le f_j(\rho_{0,j})\}\ &\text{if }\rho_{0,j}\in\Gamma_j^\mathrm{in}
		\end{cases}
		\quad\text{with }M\le \hat\eta_j'(\tilde\rho_j+).\\
		\end{split}
	\end{align}
	Since the summands in the first constraint in the original optimization problem are constant for $\tilde\rho_h$ as in \eqref{eq:HelEquivalenceRiemannSolverOptimziationProb1}, the traces $\rho_h(t,0\mp)$ do not change if we solve the standard Riemann problems with $\hat\rho_h$ satisfying
	\begin{align*}
		&\hat\eta_h'(\hat \rho_h-)\le M\le \hat\eta_h'(\hat \rho_h+)\qquad\text{for }h=1,\dots,n+m
	\end{align*}
	replacing $\tilde\rho_h$ satisfying \eqref{eq:HelEquivalenceRiemannSolverOptimziationProb4}. Therefore, we have $\rho_h(t,0)=\mathcal{RP}_h(\rho_0)$.
	
	\noindent
	\textit{Uniqueness:} \eqref{eq:HelEquivalenceRiemannSolverOptimziationProb4} uniquely characterizes the solutions $\tilde\rho_h$ to \eqref{eq:HelEquivalenceRiemannSolverOptimziationProb2}--\eqref{eq:HelEquivalenceRiemannSolverOptimziationProb23} for strictly convex $\hat\eta_h$. With the remark after \eqref{eq:HelEquivalenceRiemannSolverOptimziationProb2}, the solutions $\rho_h$ on $\Omega_h$ are uniquely defined and therefore also $\rho_h(t,0)$.
\end{proof}

\subsection{Equivalence between solutions given by Definition~\ref{def:HoldenRisebro} and Definition~\ref{def:MacroMaximumEntropyProb}}

\begin{proposition}
	Let \ref{eq:Assumptionsfh1}--\ref{eq:Assumptionsfh3}. Let $\hat g_h\in C^{0,1}( [0,1],\mathbb R)$ and let $\hat G_h$ be the entropy flux corresponding to $\hat \eta_h$ satisfying
	\begin{equation*}
		 \hat g_h(y)=
			 \pm_h \hat G_h(f_h^{-1}|_{\Gamma_h^\mathrm{out}}(y\cdot f_h^{\max})) 
\qquad\text{for a.e.~}y\in [0,1], \text{ for~} h=1,\dots,n+m.
	\end{equation*}
	Then, $\hat g_h$ is concave if and only if $\hat G_h$ is an entropy flux to a convex entropy $\hat \eta_h$.
	If we assume additionally that $\hat g_h$ is concave, then Definition~\ref{def:HoldenRisebro} and Definition~\ref{def:MacroMaximumEntropyProb} are equivalent.
\end{proposition}
\begin{proof}
	The optimization problems in Definition~\ref{def:HoldenRisebro} and Definition~\ref{def:MacroMaximumEntropyProb} differ only in their objective function. As observed in the proof of Proposition~\ref{prop:RelationRiemannProblemMaximumEntropy}, we have $\rho_h(t,0\pm_h)\in \Gamma_h^\mathrm{out}$ for the solution of Definition~\ref{def:MacroMaximumEntropyProb}. By the same arguments, this can be achieved for the solution in Definition~\ref{def:HoldenRisebro} choosing suitable $\tilde\rho_h$ such that $\tilde\rho_h=\rho_h(t,0\pm_h)$. Setting 
	\begin{alignat*}{2}
		\hat g_h(y)&=\pm_h \hat G_h(f_h^{-1}|_{\Gamma_h^\mathrm{out}}(y\cdot f_h^{\max}))	&&\qquad\text{for a.e.~}y\in[0,1]\\
		\iff \quad 
		\hat G_h(\rho)&= \pm_h\hat g_h\!\left(\frac{f_h(\rho)}{f_h^{\max}}\right)&&\qquad\text{for a.e.~}\rho\in \Gamma_h^\mathrm{out}
	\end{alignat*} 
	leads to the following relation between the traces
	\begin{equation*}
		\hat G_h(\rho_h(t,0\pm_h))=\pm_h \hat g_h\!\left(\frac{f_h(\rho_h(t,0\pm_h))}{f_h^{\max}}\right)=\pm_h \hat g_h\!\left(\frac{f_h(\rho_h(t,0))}{f_h^{\max}}\right).
	\end{equation*}
	
	The function $\hat g_h$ is concave if $\hat G_h$ is an entropy flux of a convex $\hat \eta_h$ since
	\begin{alignat*}{2}
		&\hat \eta_h'(\rho)\, f_h'(\rho)=\hat G_h'(\rho)=\pm_h \frac{f_h'(\rho)}{f_h^{\max}} \,\hat g_h'\!\left(\frac{f_h(\rho)}{f_h^{\max}}\right)\qquad &&\text{for a.e.~}\rho\in \Gamma_h^\mathrm{out},\\
		\iff\qquad &\hat \eta_h'(\rho)= \frac{\pm_h 1}{f_h^{\max}} \,\hat g_h'\!\left(\frac{f_h(\rho)}{f_h^{\max}}\right)\qquad &&\text{for a.e.~}\rho\in \Gamma_h^\mathrm{out}.
	\end{alignat*}
	The other direction is obtained by the equation above and by extending $\hat \eta_h$ to a convex function on the domain $K_h$.
\end{proof}

\section{A BGK Model on Networks}
\label{sec:KineticBGKModel}

We introduce a kinetic formulation for scalar conservation laws and a kinetic Bhatnagar--Gross--Krook--type model (BGK model), see e.g.\ \cite{Pe2002} for a reference. We introduce the equilibrium function 
\begin{gather}
	\chi(\rho,\xi)=
	\begin{cases}
		1\quad &\text{if } 0 <\xi<\rho,\\
		-1\quad &\text{if } \rho <\xi<0 ,\\
		0\quad &\text{else},
	\end{cases}	
\end{gather}
on $\mathbb R_\rho\times\mathbb R_\xi$
and the notation
\begin{align}
	D_\xi &=\{g\in\mathbb R\ |\ \sgn (\xi)\, g=|g|\le 1 \}\qquad \text{for }\xi\in\mathbb R,\\
		L^1(N\times O,D_\xi)&=\{g\in L^1(N\times O)\ |\ g(y,\xi)\in D_\xi \text{ for a.e.~}(y,\xi)\in N\times O\}
\end{align}
for measurable sets $N\subset \mathbb R^l$ and $O\subset \mathbb R$.
We recall some well-known results for the equilibrium function $\chi$ where we consider always increasing representatives of $\eta_h'$ if $\eta_h$ is convex.
\begin{lemma}[{\cite[Chapter~2]{Pe2002}}]\label{lemma:BasicConvexityProperties}
	Let $\rho\in \mathbb R$, $\eta_h,f_h\in C^{0,1}(\mathbb R)$ and let $G_{h}$ be such that $G'_{h}=\eta'_h f'_h$. 
	\begin{enumerate}[label=(\roman*)]
		\item We have 
			\begin{alignat*}{2}
				\rho =&\int_{\mathbb R} \chi(\rho,\xi)\d{\xi},\qquad & f_h(\rho) =&\int_{\mathbb R} f'_h( \xi)\; \chi(\rho,\xi)\d{\xi}+f_h(0),\\
				\eta_h(\rho) =&\int_{\mathbb R}\eta'_h(\xi)\; \chi(\rho,\xi)\d{\xi}+\eta_h(0),\qquad & G_{h}(\rho) =&\int_{\mathbb R}\eta'_h(\xi)\; f'_h( \xi)\; \chi(\rho,\xi)\d{\xi}+G_{h}(0).
			\end{alignat*}
		\item For $\eta_h$ convex, $g\in D_\xi$ and $\rho,\xi\in\mathbb R$, the sub-differential inequality
			\begin{equation*}
				\eta'_h(\xi)\, g\ge \eta'_h(\xi)\,\chi(\rho,\xi)+\eta'_h(\rho)\left(g-\chi(\rho,\xi)\right)
			\end{equation*}
			holds. 
		\item For $\eta_h$ convex, $g(\xi)=\chi(\rho,\xi)$ is the minimizer of
			\begin{equation*}
				\int_{\mathbb R}\eta'_h(\xi)\, g(\xi)\d{\xi}
			\end{equation*}
			under the constraints
			\begin{equation*}
				g\in L^1(\mathbb R,D_\xi) \quad \text{and}\quad \int_{\mathbb R}g(\xi)\d{\xi}=\rho.
			\end{equation*}
			The minimizer is unique if $\eta_h$ is strictly convex.
	\end{enumerate}
\end{lemma}

We recall the kinetic formulation for scalar conservation laws introduced in \cite{LPT1994scalar} which is equivalent to the notion of weak entropy solutions \cite{Pe2002}. $\rho_h$ is called a solution to the kinetic formulation of \eqref{eq:LWRmodel} if there exists a non-negative bounded measure $m_h(t,x,\xi)$, $m_h\in C_0( K_h,\text{weak}-M^1([0,\infty)_t\times \Omega_h))$ such that 
\begin{equation}\label{eq:KineticModel}
	\partial_t \chi(\rho_h(t,x),\xi)+f'_h(\xi)\,\partial_x\chi(\rho_h(t,x),\xi)=\partial_\xi m_h(t,x,\xi) \qquad t>0,x\in\Omega_h,\xi\in K_h,
\end{equation}
holds in the distributional sense.

\subsection{Definition of the BGK model}
Motivated by the kinetic formulation we aim to approximate $\chi(\rho_h,\xi)$ by the BGK model. We consider solutions $g_h\in L^1((0,\infty)_t\times\Omega_h\times K_h,D_\xi)$ to 
\begin{equation}\label{eq:KineticBGKModel}
	\partial_t g_h(t,x,\xi)+f'_h(\xi)\,\partial_x g_h(t,x,\xi)=\frac{\chi(\rho_{g_h}(t,x),\xi)-g_h(t,x,\xi)}{\epsilon}\qquad t>0,x\in \Omega_h,\xi\in K_h,
\end{equation}
with 
\begin{equation}\label{eq:KineticBGKModel2}
	\rho_{g_h}(t,x)=\int_{ K_h}g_h(t,x,\xi)\d{\xi},
\end{equation}
and the initial condition
\begin{equation}
	{g_h}(0,x,\xi)=\chi(\rho_{0,h}(x),\xi)\qquad \text{for a.e.~}x\in\Omega_h,\xi\in K_h,
\end{equation}
with $\rho_{0,h}\in L^1(\Omega_h,K_h)$. It remains to introduce a suitable coupling condition at the junction.
Notice that for fixed $\xi\in K_h$ the left--hand--side of \eqref{eq:KineticBGKModel} is a linear transport with characteristics $(1,f'_h(\xi))$ in the $x$-$t$-plane. 
To define a kinetic coupling condition we prescribe data for the outgoing characteristics $ g_{h}(t,0,\xi)$ with $\xi\in \Gamma_h^\mathrm{out}$ depending on the data with incoming characteristics $g_{h}(t,0,\xi)$ with $ \xi\in \Gamma_h^\mathrm{in}$.
The approach can be formalized by a kinetic coupling function
\begin{equation}\label{eq:KineticCouplingCondition}
	\Psi\colon \bigtimes_{h=1}^{n+m} L^1(\Gamma_h^{\mathrm{in}},D_\xi)\to \bigtimes_{h=1}^{n+m} L^1(\Gamma_h^{\mathrm{out}},D_\xi)
\end{equation}
and the coupling condition
\begin{equation}\label{eq:DefinitionCoupling}
	\Psi_h[g(t,0,\Gamma^{\mathrm{in}})](\xi):=\Psi_h[g_k(t,0,\Gamma_k^{\mathrm{in}}),1\le k\le n+m](\xi)=g_h(t,0,\xi)\quad \text{for a.e. }\xi\in\Gamma_h^{\mathrm{out}},
\end{equation}
for $h=1,\dots,n+m$. 

\subsection{Selection of kinetic coupling conditions}
We select the mapping $\Psi$ as a solution to a maximum entropy dissipation problem. Similar to the macroscopic case we may assume that the total kinetic mass is conserved in time
\begin{equation}\label{eq:MassConservationKineticHelp}
	\sum_{h=1}^{n+m}\iint_{\Omega_h\times K_h}g_h(T,x,\xi)\d{x}\d{\xi}=\sum_{h=1}^{n+m}\iint_{\Omega_h\times K_h}\chi(\rho_{0,h}(x),\xi)\d{x}\d{\xi}\qquad \text{for every }T>0,
\end{equation}
for any $\rho_{0,h}\in L^1(\Omega_h,K_h)$, see also \eqref{eq:MassConservationJunction}.
As in the macroscopic setting, we assume that the total kinetic entropy w.r.t. some $(\eta_1,\dots,\eta_{n+m})$ decreases in time. We obtain
\begin{align}
	\sum_{h=1}^{n+m}\iint_{\Omega_h\times K_h}\eta_h'(\xi)\,g_h(T_2,x,\xi)\d\xi+\eta_h(0)\d x 
	\le \sum_{h=1}^{n+m}\iint_{\Omega_h\times K_h}\eta_h'(\xi)\,g_h(T_1,x,\xi)\d\xi+\eta_h(0) \d{x}
\end{align}
for every $T_2\ge T_1\ge 0$, for $\rho_{0,h}\in L^1(\Omega_h,K_h)$ and for suitable entropy pairs $(\eta_h,G_h)$ satisfying \eqref{eq:constraintEntropiesintheEnergyDecayOnTheNetwork}.
Integration by parts applied to the entropy inequality \eqref{eq:WeakEntropySolution} leads to 
\begin{align}\label{eq:EntropyDissipationJunctionKinetic}
\begin{split}
	&\sum_{i=1}^{n}\left(\int_{K_i} \eta_i'(\xi)\,f_i'(\xi)\, g_i(t,0,\xi)\d \xi + G_{i}(0)\right) \\
	&- \sum_{j=n+1}^{n+m}\left(\int_{K_j}  \eta_j'(\xi)\,f_j'(\xi)\, g_j(t,0,\xi)\d \xi +  G_{j}(0)\right) \ge 0\qquad\text{for a.e.~}t>0,
\end{split}
\end{align}
for the entropy fluxes of $\eta_h$ satisfying \eqref{eq:constraintEntropiesintheEnergyDecayOnTheNetwork}, see also \eqref{eq:EntropyDecayJunctionMacro}.
As in \cite{HHW2020}, we apply a maximum entropy dissipation problem to select a kinetic coupling function $\Psi$ based on \eqref{eq:DefinitionCoupling}, \eqref{eq:MassConservationKineticHelp} and \eqref{eq:EntropyDissipationJunctionKinetic}. Let $\hat\eta_h$ be convex entropies and let $g_h(\xi)$, $\xi\in \Gamma_h^\mathrm{in}$ be given data. We select $g_h(\xi)\in D_\xi$, $\xi\in \Gamma_h^\mathrm{out}$ by maximizing \eqref{eq:EntropyDissipationJunctionKinetic} under the constraint \eqref{eq:MassConservationKineticHelp}:
\begin{alignat}{2}\label{eq:OptimizationProblemWithoutConstraints}
	&\underset{ g_h(\xi),\ \xi\in\Gamma_h^\mathrm{out}}{\max} &\quad & \sum_{i=1}^{n}\int_{\Gamma_i^\mathrm{out}} \hat\eta_i'(\xi)\,f_i'(\xi)\, g_i(\xi)\d \xi - \sum_{j=n+1}^{n+m}\int_{\Gamma_j^\mathrm{out}} \hat\eta_j'(\xi)\,f_j'(\xi)\, g_j(\xi)\d \xi \\
	&\qquad\text{s.t.}&& \sum_{i=1}^n \int_{ K_i}f'_i(\xi)\, g_i(\xi)\d{\xi}=\sum_{j=n+1}^{n+m} \int_{ K_j}f'_j(\xi)\, g_j(\xi)\d{\xi},\nonumber\\
	&&& g_h(\xi)\in D_\xi \qquad \text{for a.e. }\xi\in\Gamma_h^\mathrm{out}.\nonumber
\end{alignat}
Then, the mapping $\Psi$ is defined by $\Psi_h[g(t,0,\Gamma^\mathrm{in})](\xi)=g_h(\xi)$, $\xi\in \Gamma_h^\mathrm{out}$ where $g_h(\xi)$ is a solution to \eqref{eq:OptimizationProblemWithoutConstraints}.
\begin{proposition}\label{prop:SolutionOptimizationWithout}
	Let \ref{eq:Assumptionsfh1}--\ref{eq:Assumptionsfh3} and \ref{eq:Assumptionsetah2} hold. Let $g_h\in L^1(\Gamma_h^\mathrm{in},D_\xi)$.
	\begin{enumerate}[label=(\roman*)]
	\item  There exists a solution to \eqref{eq:OptimizationProblemWithoutConstraints} with
		\begin{equation*}
			g_h(\xi)=\chi(\hat\rho_h,\xi) \qquad \text{for } \xi\in \Gamma_h^\mathrm{out}\quad \text{with}\quad \hat\eta'_h(\hat\rho_h-)\le M\le\hat \eta'_h(\hat\rho_h+)
		\end{equation*}
		for some $M\in \mathbb R$.  
	\item The solution is unique if $\hat\eta_h$ are strictly convex for all $h=1,\dots,n+m$. 
	\item If $\hat\eta_h$ are not strictly convex, there may exist multiple solutions. They coincide almost everywhere except of $\xi\in\Gamma_h^\mathrm{out}$, $\hat\eta'_h(\xi)= M$. 
	\end{enumerate}
\end{proposition}
\begin{proof}~
	
	\noindent
	\underline{Step 1:} We rewrite the optimization problem \eqref{eq:OptimizationProblemWithoutConstraints} as
	\begin{alignat}{2} \label{eq:PropMinimizerHelpOptimizationProblem0}
		\min & \quad & & -\sum_{h=1}^{n+m} \int_{\Gamma_h^\mathrm{out}} \hat\eta'_h(\xi) \,|f_h'(\xi)|\,g_h(\xi)\d{\xi} \\
		\text{s.t.}&&& \sum_{h=1}^{n+m} \int_{\Gamma_h^\mathrm{out}}  |f_h'(\xi)|\,g_h(\xi)\d{\xi} =\sum_{h=1}^{n+m} \int_{\Gamma_h^\mathrm{in}} |f'_h(\xi)|\, g_h(\xi)\d{\xi}\label{eq:PropMinimizerHelpOptimizationProblem1}\\
		&&&  g_h(\xi)\in D_\xi\qquad \text{for a.e. }\xi\in K_h \label{eq:PropMinimizerHelpOptimizationProblem2}
	\end{alignat}
	and define
	\begin{align*}
		&g_h^M(\xi)=\chi(\rho_h^M,\xi)
			\qquad \text{for } \xi\in\Gamma_h^\mathrm{out}\quad\text{with}\quad \hat\eta'_h(\rho_h^M-)\le M\le \hat\eta'_h(\rho_h^M+).
	\end{align*}  
	For $M\in\mathbb R $ with $\mathcal L(\{\hat\eta'_h=M\})>0$ for some $h=1,\dots,n+m$, we have some freedom in the definition of $\rho_h^M(\xi)$ with $\hat\eta'_h(\rho_h^M)\le M\le \hat\eta'_h(\rho_h^M+)$. Here $\mathcal L$ denotes the one-dimensional Lebesque measure. Furthermore, notice that the left--hand--side of \eqref{eq:PropMinimizerHelpOptimizationProblem1} with $g_h=g_h^M$ is increasing w.r.t.~$M\in\mathbb R$ and it has discontinuities only if $\mathcal L(\{\hat\eta'_h=M\})>0$ for some $h=1,\dots,n+m$. 
	
	\noindent
	\underline{Step 2:}
	We aim to select $M\in \mathbb R$ and $\rho_h^M$ such that \eqref{eq:PropMinimizerHelpOptimizationProblem1} holds with $g_h=g_h^M$ on $\Gamma_h^\mathrm{out}$.
	Notice that
	\begin{align*}
		\sum_{h=1}^{n+m} \int_{\Gamma_h^\mathrm{out}} |f_h'(\xi)|\, g_h^M(\xi)\d{\xi}-\sum_{h=1}^{n+m} \int_{\Gamma_h^\mathrm{in}} |f'_h(\xi)|\, g_h(\xi)\d{\xi}\quad 
		\begin{cases}
\le 0\quad &	\text{for sufficiently small $M$},\\
\ge 0\quad &	\text{for sufficiently large $M$}.
\end{cases}
	\end{align*}
	We apply the intermediate value theorem to the left--hand--side of \eqref{eq:PropMinimizerHelpOptimizationProblem1} as a function of $M$. Notice that the function is monotone and of bounded variation (not necessarily continuous). We obtain an $M\in \mathbb R$ such that
	\begin{equation*}
		\lim_{\tilde M\to M-}\sum_{h=1}^{n+m} \int_{\Gamma_h^\mathrm{out}} |f_h'(\xi)|\, g_h^{\tilde M}(\xi)\d{\xi}\le \sum_{h=1}^{n+m} \int_{\Gamma_h^\mathrm{in}} |f'_h(\xi)|\, g_h(\xi)\d{\xi}\le \lim_{\tilde M\to M+}\sum_{h=1}^{n+m} \int_{\Gamma_h^\mathrm{out}} |f_h'(\xi)|\, g_h^{\tilde M}(\xi)\d{\xi}.
	\end{equation*}
	We fix this $M\in \mathbb R$.
	By Beppo Levi's lemma, $\lim_{\tilde M\to M\pm}g_h^{\tilde M}(\xi)$ are monotone sequences converging in $L^1(\mathbb R_\xi)$ to an equilibrium function $\chi(\rho_h^{M\pm},\xi)$ with $\hat\eta_h(\rho_h^{M\pm}-)\le M\le \hat\eta_h(\rho_h^{M\pm}+)$. We apply the intermediate value theorem again to the left--hand--side of \eqref{eq:PropMinimizerHelpOptimizationProblem1} as a function of $\rho_h^M\in [\rho_h^{M-},\rho_h^{M+}]$. This leads to the existence of $\rho_h^M$ such that \eqref{eq:PropMinimizerHelpOptimizationProblem1} holds.

	\noindent
	\underline{Step 3:} For every $g_h$ and $g_h^M$ with $\rho_h^M$ such that \eqref{eq:PropMinimizerHelpOptimizationProblem1}--\eqref{eq:PropMinimizerHelpOptimizationProblem2} are satisfied, we obtain by the sub-differential inequality in Lemma~\ref{lemma:BasicConvexityProperties} that
	\begin{align}
		\begin{split}\label{eq:PropMinimizerInequality}
			&\sum_{h=1}^{n+m} \int_{\Gamma_h^\mathrm{out}} \hat\eta'_h(\xi)\, |f_h'(\xi)|\,g_h(\xi)\d{\xi} \\
			&\ge \sum_{h=1}^{n+m} \left(\int_{\Gamma_h^\mathrm{out}} \hat\eta'_h(\xi)\, |f'_h(\xi)|\,\chi(\rho_h^M,\xi)\d{\xi} + M \int_{\Gamma_h^\mathrm{out}} |f'_h(\xi)|\left( g_h(\xi)-\chi(\rho_h^M,\xi)\right)\d{\xi} \right) \\
			&= \sum_{h=1}^{n+m} \left(\int_{\Gamma_h^\mathrm{out}} \hat\eta'_h(\xi)\, |f_h'(\xi)|\,g_h^M(\xi)\d{\xi} + M \int_{\Gamma_h^\mathrm{out}}   |f_h'(\xi)|\,(g_h(\xi)-\chi(\rho_h^M,\xi))\d{\xi} \right) \\
			& = \sum_{h=1}^{n+m} \int_{\Gamma_h^\mathrm{out}} \hat\eta'_h(\xi)\, |f_h'(\xi)|\,g_h^M(\xi)\d{\xi}.
		\end{split}
	\end{align}
	Therefore, $g_h^M$ is a solution to \eqref{eq:PropMinimizerHelpOptimizationProblem0}--\eqref{eq:PropMinimizerHelpOptimizationProblem2}. The assertion \textit{(i)} follows.

	\noindent
	\underline{Step 4:} Observe that the inequality in \eqref{eq:PropMinimizerInequality} is an equality if and only if $g_h(\xi)= g_h^M(\xi)$ for a.e.~$\xi\in\Gamma_h^\mathrm{out}\cap\{\eta_h'(\xi)\neq M\}$. Otherwise the inequality is strict. On the other hand the objective function \eqref{eq:PropMinimizerHelpOptimizationProblem0} dependents only on $g_h(\xi)$ with $\xi\notin\Gamma_h^\mathrm{out}\cap\{\eta_h'(\xi)= M\}$ as long as the constraints \eqref{eq:PropMinimizerHelpOptimizationProblem1}--\eqref{eq:PropMinimizerHelpOptimizationProblem2} are satisfied.  Hence, the assertions \textit{(ii)--(iii)} follow.
\end{proof}

To handle the non-uniqueness in the case where $\hat\eta_h$ are not strictly convex, we use the functions $\Pi^M_h$ introduced in \ref{eq:DefPiM} again. We define the following coupling condition solving \eqref{eq:OptimizationProblemWithoutConstraints}.

\begin{definition}\label{def:KineticCouplingCondition}
	We define the coupling function
	\begin{equation*}
		\Psi\colon \bigtimes_{h=1}^{n+m}L^1(\Gamma_h^\mathrm{in},D_\xi)\to \bigtimes_{h=1}^{n+m}L^1(\Gamma_h^\mathrm{out},D_\xi);\quad g\mapsto \Big(\Gamma_h^\mathrm{out}\ni \xi\mapsto \chi(\hat\rho_h,\xi)\Big)_{h=1}^{n+m},
	\end{equation*}
	for $\hat\rho_h$ such that 
	\begin{align*}
		\hat\eta'_h(\hat\rho_h-)\le M \le \hat\eta'_h(\hat\rho_h+)\quad\text{and}\quad \hat\rho_h=\Pi^M_h(z)\quad\text{for } M\in\mathbb R, z\in[0,1],
	\end{align*}
	and 
	\begin{align}\label{eq:MassConservationKineticJunction}
	\begin{split}
		&\sum_{i=1}^{n}\left( \int_{\Gamma_i^{\mathrm{out}}}f'_i(\xi)\, \Psi_i[g(\Gamma^{\mathrm{in}})](\xi)\d{\xi}+\int_{\Gamma_i^{\mathrm{in}}}f'_i(\xi)\, g_i(\xi)\d{\xi}\right)\\
		&=\sum_{j=n+1}^{n+m}\left( \int_{\Gamma_j^{\mathrm{out}}}f'_j(\xi)\, \Psi_j[g(\Gamma^{\mathrm{in}})](\xi)\d{\xi}+\int_{\Gamma_j^{\mathrm{in}}}f'_j(\xi)\, g_j(\xi)\d{\xi}\right)
	\end{split}
	\end{align}
	hold. 
\end{definition}
The existence of $\Psi$ as in Definition~\ref{def:KineticCouplingCondition} is obtained by the same arguments as in Proposition~\ref{prop:SolutionOptimizationWithout} by setting $\hat\rho_h=\Pi^M_h(z)$, $z\in [0,1]$ if $\mathcal L(\{\hat\eta'_h=M\})>0$ for some $h=1,\dots,n+m$ in Step 2.

\subsection{A kinetic \texorpdfstring{$\bf{L^1}$}{L1}--contraction property}

We prove an $L ^ 1$--contraction property of the kinetic coupling condition $\Psi$.

\begin{proposition}\label{prop:EntropyInequalityJunctionWithout}
	Let \ref{eq:Assumptionsfh1}--\ref{eq:Assumptionsfh3}, \ref{eq:Assumptionsetah1} and \ref{eq:DefPiM} hold. Let $\Psi$ be as in Definition~\ref{def:KineticCouplingCondition}.
	Then, for $g_{h}^s\in L^1(K_h,D_\xi)$ with $g_{h}^s(\xi)=\Psi_h[g^s (\Gamma^\mathrm{in})](\xi)$, $\xi\in  \Gamma_h^\mathrm{out}$, $s=1,2$:
	\begin{equation*}
		\sum_{i=1}^n \int_{ K_i}f'_i(\xi)\, |g_i^1(\xi)-g_i^2(\xi)|\d{\xi}-\sum_{j=n+1}^{n+m}\int_{ K_j}f'_j(\xi)\, |g_j^1(\xi)-g_j^2(\xi)|\d{\xi}\ge 0.
	\end{equation*}
\end{proposition}
\begin{proof}
	Observe that by definition of $\Gamma_h^\mathrm{in}$ and $\Gamma_h^\mathrm{out}$
	\begin{align}\label{eq:LemmaKruhzkovDissipationHelp1}
	\begin{split}
		&\sum_{i=1}^n \int_{ K_i}f'_i(\xi)\, |g_i^1(\xi)-g_i^2(\xi)|\d{\xi}-\sum_{j=n+1}^{n+m}\int_{ K_j}f'_j(\xi)\, |g_j^1(\xi)-g_j^2(\xi)|\d{\xi}\\
		=&\sum_{h=1}^{n+m}\left(\int_{\Gamma^\mathrm{in}_h}|f_h'(\xi)|\, |g_h^1(\xi)-g_h^2(\xi)|\d{\xi}-\int_{\Gamma^\mathrm{out}_h}|f_h'(\xi)|\, |g_h^1(\xi)-g_h^2(\xi)|\d{\xi}\right).
	\end{split}	
	\end{align}
	Due to the first constraint in \eqref{eq:OptimizationProblemWithoutConstraints}, we get
	\begin{align}\label{eq:LemmaKruhzkovDissipationHelp2}
		\sum_{h=1}^{n+m}\left(\int_{\Gamma^\mathrm{in}_h}|f_h'(\xi)| \left(g_h^1(\xi)-g_h^2(\xi) \right) \d{\xi}-\int_{\Gamma^\mathrm{out}_h}|f_h'(\xi)| \left( g_h^1(\xi)-g_h^2(\xi)\right)\d{\xi}\right)=0.
	\end{align}
	Since $g_{h}^s$ satisfy the coupling condition, there exist $\hat\rho_h^s$ as in Definition \ref{def:KineticCouplingCondition} such that $g_{h}^s(\xi)=\chi(\hat\rho_h^s,\xi)$ for a.e.~$\xi\in\Gamma_h^\mathrm{out}$, $s=1,2$. Due to the monotonicity of $\hat\eta_h'$ and $\Pi_h^M$, there exists $\mu\in\{-1,1\}$ such that $\sgn(\hat\rho_h^1-\hat\rho_h^2)=\mu$ for $h=1,\dots,n+m$ (using the convention $\sgn(0)=\mu$ throughout this proof). Since $g_{h}^s(\xi)=\chi(\hat\rho_h^s,\xi)$ for a.e.~$\xi\in\Gamma_h^\mathrm{out}$, $h=1,\dots,n+m$, we have also $\sgn(g_h^1(\xi)-g_h^2(\xi))=\mu$. 
		Combining this observation with \eqref{eq:LemmaKruhzkovDissipationHelp1} and \eqref{eq:LemmaKruhzkovDissipationHelp2} leads to
	\begin{align*}
	\begin{split}
		&\sum_{i=1}^n \int_{ K_i}f'_i(\xi)\, |g_i^1(\xi)-g_i^2(\xi)|\d{\xi}-\sum_{j=n+1}^{n+m}\int_{ K_j}f'_j(\xi)\, |g_j^1(\xi)-g_j^2(\xi)|\d{\xi}\\
		&=\sum_{h=1}^{n+m}\left(\int_{\Gamma^\mathrm{in}_h}|f_h'(\xi)|\, |g_h^1(\xi)-g_h^2(\xi)|\d{\xi}-\mu\int_{\Gamma^\mathrm{out}_h}|f_h'(\xi)|\, (g_h^1(\xi)-g_h^2(\xi))\d{\xi}\right)\\
		&=\sum_{h=1}^{n+m}\int_{\Gamma^\mathrm{in}_h}|f_h'(\xi)|\, \Big(|g_h^1(\xi)-g_h^2(\xi)|-\mu\left(g_h^1(\xi)-g_h^2(\xi)\right)\Big)\d{\xi}\ge 0.
	\end{split}	
	\end{align*}
%
\end{proof}

\section{Well-posedness of the BGK Model}

We study well-posedness of the BGK model following the presentation in \cite[Section~3.5]{Pe2002}. First, we consider a linear kinetic problem for general coupling conditions $\psi$. In a second step, we apply a fixed point argument to construct solutions to the non-linear problem with the coupling condition $\Psi$. We introduce the space $L^1_{\mu_h}(K_h)$ of measurable functions $g_h\colon K_h\to \mathbb R$ with $\int_{K_h}|g_h|\, | f_h'|\d x<\infty$. 
\begin{theorem}\label{thm:ExistenceLinearKineticModel}
Let \ref{eq:Assumptionsfh1}--\ref{eq:Assumptionsfh3} and $g_{0,h}\in L^1(\Omega_h\times K_h)$ hold. Let $m_h\in L^1((0,T)_t\times\Omega_h\times K_h)$ for all $T>0$ and let $\psi\colon \bigtimes_{h=1}^{n+m}L^1_{\mu_h}(\Gamma_h^\mathrm{in}) \to  \bigtimes_{h=1}^{n+m}L^1_{\mu_h}(\Gamma_h^\mathrm{out})$ be a continuous function. Then, there exists a unique distributional solution $g_h\in C([0,\infty)_t,L^1(\Omega_h\times K_h))$ to 
\begin{align*}
	\begin{cases}
		\partial_t g_h+f'_h(\xi)\, \partial_x g_h=\frac{m_h-g_h}{\epsilon}\quad & t>0,x\in\Omega_h,\xi\in K_h,\\
		g_h(t,0,\xi)=\psi_h[g(t,0,\Gamma^{\mathrm{in}})](\xi)\quad &t>0,\xi\in\Gamma_h^{\mathrm{out}},\\
		g_h(0,x,\xi)=g_{0,h}(x,\xi)\quad & x\in\Omega_h,\xi\in K_h.
	\end{cases}
\end{align*} 
The solution satisfies:
\begin{enumerate}[label=(\roman*)]
	\item the characteristics formula
	\begin{align*}
		&g_h(t,x,\xi)=\left [ g_{0,h}(x-f'_h(\xi)t,\xi)e^{-t/\epsilon}+\frac{1}{\epsilon}\int_0^t m_h(t-s,x-f'_h(\xi)s,\xi)\, e^{-s/\epsilon}\d{s} \right ] _{1/t>f'_h(\xi)/x}
		\\
		&\quad+\left [ g_{h}(t-x/f'_h(\xi),0,\xi)e^{-x/(f'_h(\xi)\epsilon)}+\frac{1}{\epsilon}\int_0^{x/f'_h(\xi)} m_h(t-s,x-f'_h(\xi)s,\xi)\,e^{-s/\epsilon}\d{s} \right ] _{1/t<f'(\xi)/x};
	\end{align*}
	\item the relations
	\begin{align*}
		&\partial_t\left(\sum_{h=1}^{m+n} \iint_{\Omega_h\times K_h}g_h(t,x,\xi)\d{x}\d{\xi}\right) =\sum_{h=1}^{m+n} \iint_{\Omega_h\times K_h}\frac{m_h(t,x,\xi)-g_h(t,x,\xi)}{\epsilon}\d{x}\d{\xi}\\
		&\qquad\qquad-\sum_{i=1}^n\int_{ K_i}f'_i(\xi)\, g_i(t,0,\xi)\d{\xi}+\sum_{j=n+1}^{n+m}\int_{ K_j}f'_j(\xi)\, g_j(t,0,\xi)\d{\xi},\\
		&\partial_t\left(\sum_{h=1}^{m+n} \iint_{\Omega_h\times K_h}|g_h(t,x,\xi)|\d{x}\d{\xi}\right)\le\sum_{h=1}^{m+n} \iint_{\Omega_h\times K_h}\frac{|m_h(t,x,\xi)|-|g_h(t,x,\xi)|}{\epsilon}\d{x}\d{\xi}\\
		&\qquad\qquad-\sum_{i=1}^n\int_{ K_i}f'_i(\xi)\, |g_i(t,0,\xi)|\d{\xi}+\sum_{j=n+1}^{n+m}\int_{ K_j}f'_j(\xi)\, |g_j(t,0,\xi)|\d{\xi},\\
	&\|g_h(t,x,\xi)\|_{L^\infty} \le \max \Big\{ \|g_{0,h}\|_{L^\infty},\|m_h\|_{L^\infty},\|\psi_h\|_{L^\infty}\Big\}\in [0,\infty]; 
	\end{align*}
		\item and the equation
		\begin{equation*}
			\partial_t|g_h|+f'_h(\xi)\, \partial_x|g_h|=\frac{\sgn(g_h)\, m_h-|g_h|}{\epsilon}\le \frac{| m_h|-|g_h|}{\epsilon}\qquad t>0,x\in \Omega_h,\xi\in K_h.
		\end{equation*}
\end{enumerate}
\end{theorem}
\begin{proof}
	We adapt the proof in \cite[Theorem~3.5.1]{Pe2002} to the network case. The characteristics formula~\textit{(i)} can be obtained by solving a linear ordinary differential equation along the characteristics. Notice that the continuity of $\psi$ implies that $\psi[g(t,0,\Gamma^\mathrm{in})]$ is independent of $g_h(t,0,\xi)$ with $f_h'(\xi)=0$. Therefore, we can define, e.g. $g_h(t,0,\xi)=0$ for a.e.\ $f_h'(\xi)=0$, $t>0$. The continuity of $\psi$ ensures also $\int_0^T\int_{K_h}|g_h(x\! =\! 0)|\,|f_h'(\xi)|\d{\xi}\d{t}<\infty$. Now, an approximation argument based on smooth solutions can be used to rigorously prove the existence. The relations in~\textit{(ii)} follow from integration of the characteristics formula. The equation in~\textit{(iii)} can be proven as in the case of the standard Cauchy problem. 

	It remains to prove the uniqueness. Assume that $g^s_h$, $s=1,2$ are solutions with the same initial data and the same coupling condition $\psi$. Then, $g=g^1-g^2$ is also a solution to the linear problem with $m_h=0$ and $g_h(0,x , \xi)=0$ for $x\in\Omega_h,\xi\in K_h$. We solve the auxiliary problem
	\begin{alignat*}{2}
		&\partial_t\varphi_h+f'_h(\xi)\,\partial_x\varphi_h=\frac{\varphi_h}{\epsilon}&\qquad &t\in [0,T], x\in\Omega_h,\xi\in\Gamma_h^\mathrm{in},\\
		& \varphi_h(T,x,\xi)=\nu_h(x,\xi)\in C^1_\mathrm{c}(\Omega_h\times \Gamma_h^\mathrm{in}) &&  x\in\Omega_h,\xi\in\Gamma_h^\mathrm{in},\\
		& \varphi_h(t,0,\xi)=\vartheta_h(t,\xi)\in C^1_\mathrm{c}((0,T)_t\times \Gamma_h^\mathrm{in}) && t\in [0,T], \xi\in\Gamma_h^\mathrm{in}.
	\end{alignat*}
	Since $\varphi_h$ is of class $C^1$, we are able to use it as a test function in the weak formulation of $g_h$ and obtain
	\begin{align*}
		\iint_{\Omega_h\times\Gamma_h^\mathrm{in}}g_h(T,x,\xi)\,\nu_h\d x\d\xi + \iint_{(0,T)_t\times\Gamma_h^\mathrm{in}}|f'_h(\xi)|\,g_h(t,0,\xi)\,\vartheta_h\d x\d\xi=0.
	\end{align*}
	We can take arbitrary test functions $\nu_h,\vartheta_h$ and get $g_h(T,x,\xi)=0$
 for a.e.~$x\in \Omega_h,\xi\in\Gamma_h^\mathrm{in}$	 and $g_h(t,0,\xi)=0$ for a.e. $t\in(0,T),\xi\in\Gamma_h^\mathrm{in}$. Since we have $g_h^1(t,0,\xi)=g_h^2(t,0,\xi)$ for a.e.~$t\in(0,T),\xi\in\Gamma_h^\mathrm{in}$, we get $g_h^1(t,0,\xi)=g_h^2(t,0,\xi)$ for a.e.~$t\in(0,T),\xi\in\Gamma_h^\mathrm{out}$ by the coupling condition $\psi$. Using a similar construction for $\xi\in\Gamma_h^\mathrm{out}$ leads to $g^1_h=g^2_h$ for a.e.~$t\in (0,T),x\in \Omega_h,\xi\in K_h$.
 \end{proof}

\begin{theorem}\label{thm:ExistenceBGKKineticModel}
Let \ref{eq:Assumptionsfh1}--\ref{eq:Assumptionsfh3}, \ref{eq:Assumptionsetah1} and \ref{eq:DefPiM} hold. Let $g_{0,h}\in L^1(\Omega_h\times K_h,D_\xi)$ with $\rho_{0,h}(x)=\int_{K_h}g_{0,h}(x,\xi)\d\xi$. Then, there exists a unique distributional solution $g_h\in C([0,\infty)_t,L^1(\Omega_h\times K_h,D_\xi))$ to 
\begin{align*}
	\begin{cases}
		\partial_t g_h+f'_h(\xi)\partial_x g_h=\frac{\chi(\rho_{g_h},\xi)-g_h}{\epsilon}\quad & t>0,x\in\Omega_h,\xi\in K_h,\\
			\Psi_h[g(t,0,\Gamma^{\mathrm{in}})](\xi)=g_h(t,0,\xi)\quad &t>0,\xi\in\Gamma_h^{\mathrm{out}},\\
		g_h(0,x,\xi)=g_{0,h}(x,\xi)\quad & x\in\Omega_h,\xi\in K_h,
	\end{cases}\end{align*} 
	with $\rho_{g_h}(t,x)=\int_{ K_h}g_h(t,x,\xi)\d{\xi}\in K_h$.
The solution satisfies
\begin{enumerate}[label=(\roman*)]
	\item Conservation of mass:
	\begin{align*}
		\sum_{h=1}^{n+m} \iint_{\Omega_h\times K_h}g_h(T,x,\xi)\d{x}\d{\xi}&=\sum_{h=1}^{n+m} \iint_{\Omega_h\times K_h}g_{0,h}(x,\xi)\d{x}\d{\xi}\qquad \text{for every } T>0;
	\end{align*}
	
	\item $L^1$-contraction: For two solutions $\rho_h^s$ with initial data $g_{0,h}^s$, $s=1,2$, we have
	\begin{equation*}
		\sum_{h=1}^{n+m}\|\rho^1_h(t)-\rho^2_h(t)\|_{ L^1(\Omega_h)}\le \sum_{h=1}^{n+m}\|g^1_h(t)-g^2_h(t)\|_{L^1(\Omega_h\times K_h)}\le \sum_{h=1}^{n+m}\|g_{0,h}^1-g_{0,h}^2\|_{L^1(\Omega_h\times K_h)};
	\end{equation*}
	\item There exists a non-negative bounded function $m_h=m_h(t,x,\xi)$ such that 
	\begin{align*}
		\frac{\chi(\rho_{g_h},\xi)-g_h}{\epsilon}=\partial_\xi m_h \qquad  t>0,x\in\Omega_h,\xi\in K_h,
	\end{align*}
	holds 
	and there exist functions $\mu_h\in L^\infty(K_h)$ independent of $\epsilon$ such that 
	\begin{equation*}
		\iint_{(0,T)_t\times\Omega_h}m_h(t,x,\xi)\d t\d x\le \mu_h(\xi)\qquad \text{for a.e.~}\xi\in K_h.
	\end{equation*} 
\end{enumerate}
\end{theorem}
\begin{proof}
The proof is based on arguments in \cite[Theorem~3.6.1]{Pe2002} and \cite[Theorem~2.1]{Ho2020}.
	We define the Banach spaces $X_T=C([0,T]_t,\bigtimes_{h=1}^{n+m}L^1(\Omega_h,K_h))$ and $Y_T=C([0,T]_t,\bigtimes_{h=1}^{n+m}L^1(\Omega_h\times K_h,D_\xi))$ with their standard norms. 
	We define the map
\begin{equation*}
	\Phi(v)=\rho\quad\text{with}\quad\rho_h=\int_{ K_h}g_h\d{\xi},
\end{equation*} 
for the unique solution $g$ obtained by Theorem~\ref{thm:ExistenceLinearKineticModel} with $m_h(t,x,\xi)=\chi(v_h(t,x),\xi)$, $v\in X_T$.
The statement of Theorem~\ref{thm:ExistenceLinearKineticModel} holds true for $\Psi$ as in Definition~\ref{def:KineticCouplingCondition} and we obtain $g\in Y_T$ and $\rho\in X_T$. This can be shown by a careful treatment of the set $D_\xi$ involved in the definition of $\Psi, X_T, Y_T$ and by using the characteristics formula.

\noindent
\underline{Step 1:} The operator $\Phi$ is a strict contraction: Take $v^s\in X_T$, $s=1,2$. Due to Theorem~\ref{thm:ExistenceLinearKineticModel} and Proposition~\ref{prop:EntropyInequalityJunctionWithout}, we have
\begin{align*}
 	&\partial_t \sum_{h=1}^{m+n}\iint_{\Omega_h\times K_h}|g_h^1-g_h^2|\d{x}\d\xi\\
 	&\le \frac{1}{\epsilon} \sum_{h=1}^{m+n}\iint_{\Omega_h\times K_h}|\chi(v^1_h,\xi)-\chi(v^2_h,\xi)|-|g_h^1-g_h^2|\d{x}\d\xi\\
 	&\qquad -\sum_{i=1}^n\int_{ K_i}f'_i(\xi)\, |g_i^1(t,0,\xi)-g_i^2(t,0,\xi)|\d{\xi}+\sum_{j=n+1}^{n+m}\int_{ K_j}f'_j(\xi)\, |g_j^1(t,0,\xi)-g_j^2(t,0,\xi)|\d{\xi}\\
 	&\le  \frac{1}{\epsilon} \sum_{h=1}^{m+n}\int_{\Omega_h}\left[ |v^1_h-v^2_h|-\int_{ K_h}|g_h^1-g_h^2|\d\xi\right] \d{x}.
\end{align*} 
The differential inequality implies a Gronwall estimate on $g_h^s$
\begin{equation*}
	\|g^1-g^2\|_{Y_T}=\sum_{h=1}^{m+n}\iint_{\Omega_h\times K_h}|g_h^1-g_h^2|\d{x}\d\xi\le (1-e^{-T/\epsilon})\|v^1-v^2\|_{X_T}
\end{equation*}
 and therefore also on $\rho^s$ 
\begin{equation*}
 	\|\rho^1-\rho^2\|_{X_T}\le (1-e^{-T/\epsilon})\|v^1-v^2\|_{X_T}.
\end{equation*}

\noindent
\underline{Step 2:} Existence and uniqueness: We apply the Banach fixed point theorem to obtain a unique fixed point $\rho$ of $\Phi$. The unique kinetic solution $g_h$ is obtained by Theorem~\ref{thm:ExistenceLinearKineticModel} with $m_h(t,x,\xi)=\chi(\rho_h(t,x),\xi)$. A unique global (in time) solution is obtained by a standard construction \cite{Pe2002}. 

\noindent
\underline{Step 3:} The remaining properties follow from Theorem~\ref{thm:ExistenceLinearKineticModel}, Proposition~\ref{prop:EntropyInequalityJunctionWithout} and \eqref{eq:MassConservationKineticJunction}. Note that $\chi(\rho_{g_h},\xi)=0$ for $\xi\notin K_h$, $\int_{ K_h}|\chi(\rho_{g_h}^1,\xi)-\chi(\rho_{g_h}^2,\xi)|\d\xi=|\rho_{g_h}^1-\rho_{g_h}^2|\le \int_{ K_h}|g_{h}^1-g_h^2|\d\xi$ and $\rho_{g_h}\in K_h$ for $g_h\in L^1(K_h,D_\xi)$. The existence of $m_h$ in \textit{(iii)} follows as in \cite[Corollary~3.6.2]{Pe2002}.
\end{proof}

\subsection{Stationary solutions to the BGK model}

We prove existence of a class of stationary solutions to the BGK model on networks. The stationary solutions will be used to justify the convergence towards the macroscopic coupling condition.

We define the set
	\begin{equation}\label{eq:defmathcalK}
		\mathcal K=\left\{\bar\rho \in \bigtimes_{h=1}^{n+m}K_h\,\Big|\, \bar\rho=\mathcal{RP}(\bar\rho) \text{ with }\bar\rho_h\in\Gamma_h^\mathrm{out} \text{ if } f_h(\bar\rho_h)=f_h(\hat\rho_h), \hat\rho_h\in\Gamma_h^\mathrm{out}\right\}
	\end{equation}
where $\hat\rho_h$ is defined as in Definition~\ref{def:RiemannSolver} with initial data $\rho_{0,h}=\bar\rho_h$.

\begin{proposition}\label{prop:StationaryKineticSolution}
	Let \ref{eq:Assumptionsfh1}--\ref{eq:Assumptionsfh4}, \ref{eq:Assumptionsetah1} and \ref{eq:DefPiM} hold. For every $\bar \rho\in\mathcal K$ and $\epsilon>0$, there exists a unique stationary solution $\bar g_h^\epsilon(t,x,\xi)=\bar g_h^\epsilon(x,\xi)$ to the BGK model in Theorem~\ref{thm:ExistenceBGKKineticModel} such that
	\begin{equation*}
		\bar g_h^\epsilon(x,\cdot) \to \chi(\bar\rho_h,\cdot)\qquad\text{in }L^1(K_h)\text{ as }|x|\to\infty, x\in\Omega_h.
	\end{equation*}
	Furthermore, we have
	\begin{equation*}
		\bar g_h^\epsilon(t,x,\xi)\to \chi(\bar\rho_h,\cdot)\qquad \text{in }L^1_{\mathrm{loc}}((0,\infty)_t\times\Omega_h\times K_h)\text{ as }\epsilon\to 0.
	\end{equation*}
\end{proposition}
\begin{proof}
	For notational convenience we use the notation $\sgn_+(\xi)=\max\{\sgn(\xi),0\}$, $\sgn_-(\xi)=\min\{\sgn(\xi),0\}$ and $\sigma_h=\operatorname{argmax}f_h(\Gamma_h^\mathrm{out})$. Let $\hat\rho_h$ be defined as in Definition~\ref{def:RiemannSolver} with initial data $\rho_{0,h}=\bar\rho_h$.
	
	For $i=1,\dots,n$ we compute
	\begin{align*}
		\int_{\Gamma_i^\mathrm{out}}f_i'(\xi)\, \chi(\hat\rho_i,\xi)\d\xi 
		&=\int_{\Gamma_i^\mathrm{out}}f_i'(\xi)\, \sgn_-(\xi)\d\xi+ \int_{\Gamma_i^\mathrm{out}}f_i'(\xi)\, \mathbf{1}_{(-\infty,\hat\rho_i]}\d\xi \\
		&=\sgn_+(0-\sigma_i)(f_i^{\max}-f_i(0))+\sgn_+(\hat\rho_i-\sigma_i)(f_i(\hat\rho_i)-f_i^{\max}),\\ 
		\int_{\Gamma_i^\mathrm{in}}f_i'(\xi)\, \sgn_+(\xi)\d\xi &=\sgn_+(\sigma_i-0)(f_i^{\max}-f_i(0)),\\
		\int_{\Gamma_i^\mathrm{in}}f_i'(\xi)\, \sgn_-(\xi)\d\xi &=-f_i^{\max}+\sgn_+(\sigma_i-0)(f_i^{\max}-f_i(0)),
	\end{align*}
		and obtain 
	\begin{align*}
		\int_{\Gamma_i^\mathrm{out}}f_i'(\xi)\, \chi(\hat\rho_i,\xi)\d\xi +\int_{\Gamma_i^\mathrm{in}}f_i'(\xi)\, \sgn_+(\xi)\d\xi+f_i(0) &= f_i^{\max}+\sgn_+(\hat\rho_i-\sigma_i)(f_i(\hat\rho_i)-f_i^{\max})\ge f_i(\bar\rho_i),\\
		\int_{\Gamma_i^\mathrm{out}}f_i'(\xi)\, \chi(\hat\rho_i,\xi)\d\xi  +\int_{\Gamma_i^\mathrm{in}}f_i'(\xi)\, \sgn_-(\xi)\d\xi +f_i(0) &= \sgn_+(\hat\rho_i-\sigma_i)(f_i(\hat\rho_i)-f_i^{\max})\le 0\le f_i(\bar\rho_i).
	\end{align*}
	By the same arguments we get 
	\begin{align*}
		\int_{\Gamma_j^\mathrm{out}}f_j'(\xi)\, \chi(\hat\rho_j,\xi)\d\xi  +\int_{\Gamma_j^\mathrm{in}}f_j'(\xi)\, \sgn_-(\xi)\d\xi +f_j(0) &= f_j^{\max}+\sgn_+(\sigma_j-\hat\rho_j)(f_j(\hat\rho_j)-f_j^{\max})\ge f_j(\bar\rho_j),
		\\
		\int_{\Gamma_j^\mathrm{out}}f_j'(\xi)\, \chi(\hat\rho_j,\xi)\d\xi +\int_{\Gamma_j^\mathrm{in}}f_j'(\xi)\, \sgn_+(\xi)\d\xi+f_j(0) &= \sgn_+(\sigma_j-\hat\rho_j)(f_j(\hat\rho_j)-f_j^{\max})\le 0\le f_j(\bar\rho_j)	,
	\end{align*}
	for $j=n+1,\dots,n+m$.

	Using these estimates and \ref{eq:Assumptionsfh4}, we are able to repeat the arguments from \cite[Proposition~3]{Va2012} in our setting. 
	For $i=1,\dots,n$ we  set $u_i=\sigma_i-\rho_i$, replace the Burgers' flux by $\tilde f_i(u_i)=f_i(\sigma_i-u_i)-f_i^{\max}$ and use the variable transformation $y \mapsto -x$.
	For $j=n+1,\dots,n+m$ we set $u_j=\sigma_j-\rho_j$ and $\tilde f_j(u_j)=f_j^{\max}-f_j(\sigma_j-u_j)$ without using a variable transformation.

	 Therefore, there exists a unique stationary solution $\bar g_h(t,x,\xi)=\bar g_h(x,\xi)$ to the BGK model on $\Omega_h$ satisfying 
	\begin{align*}
	\begin{cases}
		f'_h(\xi)\, \partial_x \bar g_h^\epsilon=\frac{\chi(\rho_{\bar g_h^\epsilon},\xi)-\bar g_h^\epsilon}{\epsilon}\quad & t>0,x\in\Omega_h,\xi\in K_h,\\
			\bar g_h^\epsilon(0,\xi)=\chi(\hat\rho_h,\xi)\quad &t>0,\xi\in\Gamma_h^{\mathrm{out}},\\
		\int_{K_h}f_h'(\xi)\, \bar g_h^\epsilon(x,\xi)\d\xi +f_h(0)=f_h(\bar\rho_h)\quad & x\in\Omega_h,\\
		\bar g_h^\epsilon(x,\cdot) \to \chi(\bar\rho_h,\cdot) &\text{in }L^1(K_h)\text{ as }|x|\to\infty, x\in\Omega_h.
	\end{cases}
	\end{align*}
	Remark that Proposition~3 in \cite{Va2012} requires the assumption $\bar\rho\in\mathcal K$ since
	\begin{align*}
		\int_{\Gamma_i^\mathrm{out}}f_i'(\xi)\, \chi(\hat\rho_i,\xi)\d\xi  +\int_{\Gamma_i^\mathrm{in}}f_i'(\xi)\, \sgn_+(\xi)\d\xi +f_i(0) &=  f_i(\bar\rho_i),\\
		\text{resp.}\quad \int_{\Gamma_j^\mathrm{out}}f_j'(\xi)\, \chi(\hat\rho_j,\xi)\d\xi  +\int_{\Gamma_j^\mathrm{in}}f_j'(\xi)\, \sgn_-(\xi)\d\xi +f_j(0) &=  f_j(\bar\rho_j),
	\end{align*}
	if $f_h(\hat\rho_h)=f_h(\bar\rho_h)$ for some $h=1,\dots,n+m$.
	
	The functions $\bar g_h(t,x,\xi)=\bar g_h(x,\xi)$ satisfy the coupling condition at the junction since $\bar g_h^\epsilon(t,0,\xi)=\chi(\hat\rho_h,\xi)$ for a.e.\ $t>0,\xi\in\Gamma_h^{\mathrm{out}}$ and
	\begin{equation*}
		\sum_{i=1}^{n}\int_{K_i}f_i'(\xi)\, \bar g_i(t,0,\xi) \d\xi-\sum_{j=n+1}^{n+m}\int_{K_j}f_j'(\xi)\, \bar g_j(t,0,\xi) \d\xi=\sum_{i=1}^{n}f_i(\bar\rho_i)-\sum_{j=n+1}^{n+m}f_j(\bar\rho_j)=0.
	\end{equation*}
	
	The convergence $\bar g_h^\epsilon\to \chi(\bar\rho_h,\cdot)$ as $\epsilon\to 0$ follows from $\bar g_h^{\epsilon_1}(\epsilon_1\, x,\xi)=\bar g_h^{\epsilon_2}(\epsilon_2\, x,\xi)$ for a.e.~$x\in\Omega_h,\xi\in K_h$ and $\bar g_h^\epsilon(x,\cdot) \to \chi(\bar\rho_h,\cdot)$ as $|x|\to\infty$. 
\end{proof}

\section{Existence and Uniqueness of the LWR model at the junction}

\subsection{Interior Relaxation}

We apply a compactness result by Panov \cite{Pa1995} for measure--valued solutions to pass to the limit in the interior of the domains $\Omega_h$. 
The compactness result uses a regularizing effect of non-linear scalar conservation laws. The regularizing effect is only available if the non--degeneracy condition \ref{eq:Assumptionsfh4} holds. Let us remark that compactness results with possibly degenerate flux are available for the initial value problem \cite{Pe2002} and initial boundary value problem \cite{IV2004,Kw2008}. For these results, a careful analysis for the initial and boundary conditions is needed. The extension of them to networks may be a subject for future research.

\begin{theorem}\label{thm:CompactnessBGKSolutions}
	Let \ref{eq:Assumptionsfh1}--\ref{eq:Assumptionsfh4}, \ref{eq:Assumptionsetah1} and \ref{eq:DefPiM} hold. Then, there exist $\rho_h\in C([0,\infty)_t,L^1(\Omega_h,K_h))$ such that, after possibly taking a subsequence,
	\begin{alignat*}{2}
		\int_{ K_h}g_h^\epsilon\d\xi =\rho_{g_h^\epsilon}&\to \rho_h \qquad \text{as }\epsilon\to 0 \quad \text{in }L^1_{\mathrm{loc}}([0,\infty)_t\times\Omega_h).
	\end{alignat*}
	Furthermore, $\rho_h$ solve the entropy inequality \eqref{eq:WeakEntropySolution} on $(0,\infty)_t\times\Omega_h$ in the distributional sense. 
\end{theorem}
\begin{proof}
	Since $g_h\in L^1((0,\infty)_t\times\Omega_h\times K_h,D_\xi)$,  
	 $\nu_h^\epsilon=\partial_\xi g_h^\epsilon$ defines a measure in $\xi$ for a.e. $t\in (0,\infty),x\in\Omega_h$.
	 $\nu_h^\epsilon$ are measure--valued solutions to the LWR model in the sense of
	\begin{align*}
		&\partial_t \left(\int_{K_h}|\xi-k| \d{\nu_h^\epsilon(\xi)}\right)+\partial_x\left(\int_{K_h}\sgn(\xi-k)\,(f_h(\xi)-f_h(k))\, \d{\nu_h^\epsilon(\xi)}\right)\\
		&=\partial_t \left(\int_{K_h}\sgn(\xi-k)\, g_h^\epsilon(\xi)\d{\xi}\right)+\partial_x\left(\int_{K_h}\sgn(\xi-k)\,f_h'(\xi)\, g_h^\epsilon(\xi)\d{\xi}\right)\\
		&=\int_{K_h}\sgn(\xi-k)\,\frac{\chi(\rho_{g_h^\epsilon},\xi)-g_h^\epsilon}{\epsilon}\d\xi\le 0
	\end{align*}
	for each $k\in K_h$. 
	Furthermore, $\nu_h^\epsilon$, $\epsilon>0$ are bounded sequences of measure--valued functions since $\operatorname{supp} \nu_h^\epsilon(t,x)\subset K_h$.
	
	By \cite[Theorem~5]{Pa1995}, there exist measure--valued functions $\nu_h$ such that, after possibly taking a subsequence, $\nu_h^\epsilon\to \nu_h$ strongly in the sense of measure--valued functions. We conclude $\rho_{g_h^\epsilon}=\int_{K_h}\xi\d{\nu_h^\epsilon(\xi)}\to\int_{K_h}\xi\d{\nu_h(\xi)}= \rho_h$ in $L^1_\mathrm{loc}([0,\infty)_t\times\Omega_h)$. By standard arguments \cite[Lemma~3.3]{PeTa1991}, it can be shown that $\rho_h$ satisfies the entropy inequality \eqref{eq:WeakEntropySolution}. The time continuity follows by Theorem~\ref{thm:HelpExistenceTrace}.
\end{proof}

\subsection{Convergence at the Junction}

\begin{proposition}
	Let \ref{eq:Assumptionsfh1}--\ref{eq:Assumptionsfh4}, \ref{eq:Assumptionsetah1} and \ref{eq:DefPiM} hold. Let $g_h^\epsilon$ and $\rho_h$ be the solutions obtained in Theorem \ref{thm:ExistenceBGKKineticModel} and Theorem \ref{thm:CompactnessBGKSolutions}. Let $\mathcal K$ be as in \eqref{eq:defmathcalK}. Then, there exists a set of measure zero $\mathcal N\subset (0,\infty)$ such that 
	\begin{align*}
		&\sum_{i=1}^n G_i(\rho_i(t,0),\bar\rho_i)-\sum_{j=n+1}^{n+m}G_j(\rho_j(t,0),\bar\rho_j)\ge 0\qquad\text{for all }t\in(0,\infty)\backslash\mathcal N,
	\end{align*}
	for every $\bar\rho\in \mathcal K$, where $G_h(\rho_1,\rho_2)$ denotes the Kru\v{z}kov entropy flux defined in \eqref{eq:DefKruzkovEntropyPair}.
\end{proposition}
\begin{proof}~

	\noindent
	\underline{Step 1:} We prove that $g_h^{\epsilon}-\chi(\rho_{g_h^\epsilon},\xi)\to 0$ for a.e.~$t,x,\xi$: Take a strictly convex $\eta\colon\mathbb R\to \mathbb R$. Using the sub-differential inequality in Lemma \ref{lemma:BasicConvexityProperties} gives
	\begin{align}\label{eq:HelpProofBoundRelax1}
		 (\eta' (\xi)-\eta '(\rho_{g_h^\epsilon}))\left(g_h^{\epsilon}-\chi(\rho_{g_h^\epsilon},\xi)\right)\ge 0\quad\text{for a.e.~}t,x,\xi.
	\end{align}
	On the other hand we have
	\begin{align}\label{eq:HelpProofBoundRelax2}
	\begin{split}
	 	& \iiint_{(0,\infty)_t\times\Omega_h\times K_h} (\eta' (\xi)-\eta '(\rho_{g_h^\epsilon}))\left(g_h^{\epsilon}-\chi(\rho_{g_h^\epsilon},\xi)\right)\d t\d x\d\xi\\
	 	&=\iiint_{(0,\infty)_t\times\Omega_h\times K_h} \eta' (\xi)\left(g_h^{\epsilon}-\chi(\rho_{g_h^\epsilon},\xi)\right)\d t\d x\d\xi\\
		&=\epsilon  \iiint_{(0,\infty)_t\times\Omega_h\times K_h} \eta' (\xi)\left(\partial_t g_h^\epsilon +f_h'(\xi)\,\partial_x g_h^\epsilon\right)\d t\d x\d\xi \\
		&=\epsilon \iint_{\Omega_h\times K_h} \eta' (\xi)\left(g_h^\epsilon(T,x,\xi)-g_h^\epsilon(0,x,\xi)\right) \d x\d\xi \pm_h \epsilon  \iint_{(0,\infty)_t\times K_h} \eta' (\xi)\,f_h'(\xi)\, g_h^\epsilon(t,0,\xi)\d t\d\xi\\
		&\le C \epsilon 
	\end{split}
	\end{align}
	with a constant $C>0$ independent of $\epsilon>0$. Since $\eta$ is strictly convex, we have $\eta' (\xi)\neq \eta '(\rho_{g_h^\epsilon})$ for a.e.~$\xi\in K_h$. By \eqref{eq:HelpProofBoundRelax1}--\eqref{eq:HelpProofBoundRelax2}, we obtain $g_h^{\epsilon}-\chi(\rho_{g_h^\epsilon},\xi)\to 0$ for a.e.~$t,x,\xi$.

\noindent
\underline{Step 2:}
Fix $\bar\rho\in\mathcal K$ and let $\bar g_h^\epsilon$ by as in Proposition~\ref{prop:StationaryKineticSolution}.
Since \textit{(iii)} in Theorem~\ref{thm:ExistenceLinearKineticModel} and $|\chi(\rho_{g_h^\epsilon},\xi)-\chi(\rho_{\bar g_h^\epsilon},\xi)|\le |g_h^\epsilon-\bar g_h^\epsilon|$, we have
\begin{equation*}
	\sum_{h=1}^{n+m}\iiint_{(0,\infty)_t\times \Omega_h\times K_h} |g_h^\epsilon-\bar g_h^\epsilon|\,\partial_t\phi_h+f_h'(\xi)\,|g_h^\epsilon-\bar g_h^\epsilon|\,\partial_x\phi_h  \d t\d x\d\xi \ge 0
\end{equation*}
for $\phi_h\in \mathcal D((0,\infty)_t\times \Omega_h)$, $\phi_h\ge 0$. We aim to extend this inequality to $\phi_h \in\mathcal D((0,\infty)_t\times\mathbb R_x)$, $\phi_h\ge 0$ with $\phi_h(t,0)=\phi_{h'}(t,0)$ for all $h,h'=1,\dots,n+m$. Approximating $\phi_h \in\mathcal D((0,\infty)_t\times\mathbb R_x)$ in $(0,\infty)_t\times\Omega_h $ by a suitable sequence in $ \mathcal D((0,\infty)_t\times\mathbb R_x)$ \cite[Proposition~1]{ACD2017} and using the fact that the kinetic solutions $g_h^\epsilon, \bar g_h^\epsilon$ admit strong traces at $x=0$ leads to 
\begin{align*}
	0&\le \sum_{h=1}^{n+m}\iiint_{(0,\infty)_t\times \Omega_h\times K_h} |g_h^\epsilon-\bar g_h^\epsilon|\,\partial_t\phi_h+f_h'(\xi)\,|g_h^\epsilon-\bar g_h^\epsilon|\,\partial_x\phi_h  \d t\d x\d \xi \\
	&-\sum_{i=1}^m \iiint_{(0,\infty)_t\times K_i} f_i'(\xi)\,|g_i^\epsilon-\bar g_i^\epsilon|\,\phi_i(t,0)  \d t\d \xi
	+\sum_{j=n+1}^{n+m} \iiint_{(0,\infty)_t\times K_j} f_j'(\xi)\,|g_j^\epsilon-\bar g_j^\epsilon|\,\phi_j (t,0) \d t\d \xi\\
	&\le \sum_{h=1}^{n+m}\iiint_{(0,\infty)_t\times \Omega_h\times K_h} |g_h^\epsilon-\bar g_h^\epsilon|\,\partial_t\phi_h+f_h'(\xi)\,|g_h^\epsilon-\bar g_h^\epsilon|\,\partial_x\phi_h  \d t\d x\d \xi
\end{align*}
for every $\phi_h \in\mathcal D((0,\infty)_t\times\mathbb R_x)$, $\phi_h\ge 0$ with $\phi_h(t,0)=\phi_{h'}(t,0)$ for all $h,h'=1,\dots,n+m$. Using the fact that $g_h^\epsilon-\chi(\rho_{g_h^\epsilon},\cdot)\to 0$, $\rho_{g_h^\epsilon}\to \rho_h$ and $\bar g_h^\epsilon\to \chi(\bar\rho_h,\cdot)$ leads to 
\begin{equation*}
	\sum_{h=1}^{n+m}\iiint_{(0,\infty)_t\times \Omega_h\times K_h} |\rho_h-\bar\rho_h|\,\partial_t\phi_h+G_h(\rho_h,\bar\rho_h)\,\partial_x\phi_h  \d t\d x\d \xi \ge 0
\end{equation*}
for every $\phi_h \in\mathcal D((0,\infty)_t\times\mathbb R_x)$, $\phi_h\ge 0$ with $\phi_h(t,0)=\phi_{h'}(t,0)$ for all $h,h'=1,\dots,n+m$. We set $\phi_h(t,x)=\phi^b(t)\,\phi_{h}^{a,r}(x)$ with $\phi_h^{a,r}(x)=1$ for $|x|\le r/2$, $\phi_h^{a,r}(x)=0$ for $|x|\ge r$, $|(\phi_h^{a,r})'(x)|\le C/r$ and with $\phi^b\in\mathcal D(0,\infty)$, $\phi^b\ge 0$. We take the limit $r \to 0 $ using Lebesque's theorem together with Theorem~\ref{thm:HelpExistenceTrace} and obtain
	\begin{align*}
		&\sum_{i=1}^n\int_0^\infty G_i(\rho_i(t,0),\bar\rho_i)\,\phi^b(t)\d t-\sum_{j=n+1}^{n+m}\int_0^\infty G_j(\rho_j(t,0),\bar\rho_j)\,\phi^b(t)\d t\ge 0
	\end{align*}
	for every $\phi^b\in\mathcal D(0,\infty)$, $\phi^b\ge 0$. 
	
	\noindent
	\underline{Step 3:} We proved that the desired inequalities hold for $t\in(0,\infty)\backslash\mathcal N(\bar\rho)$ with sets of measure zero $\mathcal N(\bar\rho)$. It remains to prove that $\mathcal N(\bar\rho)$ can be chosen independent of $\bar\rho\in\mathcal K$. 
	
	We take a countable and dense subset $ \mathcal K^*$ of $\mathcal K$. We set $\mathcal N=\bigcup_{\mathcal K^*}\mathcal N(\bar\rho^k)$ and obtain for all $\bar\rho^k\in  \mathcal K^*$: 
	\begin{align*}
		&\sum_{i=1}^n G_i(\rho_i(t,0),\bar\rho_i^k)-\sum_{j=n+1}^{n+m}G_j(\rho_j(t,0),\bar\rho_j^k)\ge 0\qquad\text{for all }t\in(0,\infty)\backslash\mathcal N.
	\end{align*}
	Since the entropy flux $G_h(\rho_h(t,0),\bar\rho_h)$
	depends continuously on $\bar\rho_h$,
	we obtain the result by approximating each $\bar \rho_h\in \mathcal K$ by a sequence in $\mathcal K^*$.
\end{proof}

\begin{proposition}\label{prop:EntropyCondToRiemannSolver}
	Let \ref{eq:Assumptionsfh1}--\ref{eq:Assumptionsfh3}, \ref{eq:Assumptionsetah1} and \ref{eq:DefPiM} hold.
	Let $\rho_h\in K_h$ be such that
	\begin{align*}
		\sum_{i=1}^n G_i(\rho_i,\bar\rho_i)-\sum_{j=n+1}^{n+m} G_j(\rho_j,\bar\rho_j)\ge 0\qquad \text{for every }\bar\rho\in \mathcal K
	\end{align*}	
	 with $\mathcal K$ as in \eqref{eq:defmathcalK}, where $G_h(\rho_1,\rho_2)$ denotes the Kru\v{z}kov entropy flux defined in \eqref{eq:DefKruzkovEntropyPair}.
	 Then, 
	 \begin{align*}
		\sum_{i=1}^n G_i(\rho_i,\tilde\rho_i)-\sum_{j=n+1}^{n+m} G_j(\rho_j,\tilde\rho_j)\ge 0\qquad\text{for every }\tilde\rho_h\in K_h\text{ with }\tilde\rho=\mathcal{RP}(\tilde\rho)
	\end{align*}	
	and $\mathcal{RP}(\rho)=\rho$.
\end{proposition}
\begin{proof}
	We follow the arguments in \cite[Lemma~2.8]{ACD2017}.
	Take $\tilde\rho$ with $\tilde\rho=\mathcal{RP}(\tilde\rho)$. Then, we define $\bar\rho\in \mathcal K$ by
	\begin{equation*}
		\bar\rho_h=
		\begin{cases}
			\hat\rho_h\quad&\text{if }\hat\rho_h\in\Gamma_h^\mathrm{out},f_h(\tilde\rho_h)=f_h(\hat\rho_h), \\
			\tilde\rho_h\quad &\text{else},
		\end{cases}
	\end{equation*}
	with $\hat\rho_h$ as in Definition~\ref{def:RiemannSolver} with initial data $\rho_{0,h}=\tilde\rho_h$.
	The first part follows from
	\begin{alignat*}{3}
		G_i(\rho_i,\bar\rho_i)&=\sgn(\rho_i-\bar\rho_i)(f_i(\rho_i)-f_i(\bar\rho_i))&&\le \sgn(\rho_i-\tilde\rho_i)(f_i(\rho_i)-f_i(\tilde\rho_i))&&=G_i(\rho_i,\tilde\rho_i),\\
		G_j(\rho_j,\bar\rho_j)&=\sgn(\rho_j-\bar\rho_j)(f_j(\rho_j)-f_j(\bar\rho_j))&&\ge \sgn(\rho_j-\tilde\rho_j)(f_j(\rho_j)-f_j(\tilde\rho_j))&&=G_j(\rho_j,\tilde\rho_j).
	\end{alignat*}

	For showing $\mathcal{RP}(\rho)=\rho$, we solve the generalized Riemann problem with initial data $\rho_{0,h}=\rho_h$ and denote the traces at the junction by $\rho^*_h=\rho_h(t,0)$. We have
	\begin{align*}
		\sum_{i=1}^n G_i(\rho_i,\rho_i^*)-\sum_{j=n+1}^{n+m} G_j(\rho_j,\rho_j^*)\ge 0.
	\end{align*}
	On the other hand we have
	\begin{equation*}
		f_i(\rho_i)-f_i(\rho_i^*)
		\begin{cases}
			\ge 0\quad &\text{if }\rho_i\le \hat\rho_i,\\
			\le 0\quad &\text{if }\rho_i\ge \hat\rho_i,\\
		\end{cases}
	\end{equation*}
	where $\hat\rho_i$ is as in Definition~\ref{def:RiemannSolver} with initial data $\rho_{0,h}=\rho_h$. Therefore, $G_i(\rho_i,\rho_i^*)=-|f_i(\rho_i)-f_i(\rho_i^*)|$ for $i=1,\dots,n$. By the same arguments we get $G_j(\rho_j,\rho_j^*)=|f_j(\rho_j)-f_j(\rho_j^*)|$ for $j=n+1,\dots,n+m$. We conclude that
	\begin{align*}
		0\le \sum_{i=1}^n G_i(\rho_i,\rho_i^*)-\sum_{j=n+1}^{n+m} G_j(\rho_j,\rho_j^*)\le 0.
	\end{align*}
	Since all summands have the same sign, we have $f_h(\rho_h)=f_h(\rho_h^*)$ for $h=1,\dots,n+m$. The result follows from the construction of $\rho_h^*$ by standard Riemann problems.
\end{proof}

\subsection{Uniqueness and \texorpdfstring{$\bf{L^1}$}{L1}--contraction property}
\begin{proposition}\label{prop:Uniqueness}
	Let \ref{eq:Assumptionsfh1}--\ref{eq:Assumptionsfh3}, \ref{eq:Assumptionsetah1} and \ref{eq:DefPiM} hold. The solutions $\rho_h$ to Definition~\ref{def:CauchyProblemMacro} are unique. For initial data $\rho_{0,h}^s\in L^1(\Omega_h,K_h)$, $s=1,2$, the contraction property
	\begin{align*}
		\sum_{h=1}^{n+m}\int_{\Omega_h}|\rho_{h}^1(T)-\rho_{h}^2(T)|\d x\le \sum_{h=1}^{n+m}\int_{\Omega_h}|\rho_{0,h}^1-\rho_{0,h}^2|\d x 
	\end{align*}
	holds for every $T\ge 0$.
\end{proposition}
\begin{proof}
	The uniqueness proof is based on Kru\v{z}kov estimates. We follow \cite[Proposition~1]{ACD2017}.
	Assume that we have two solutions $\rho_h^s$, $s=1,2$ to Definition~\ref{def:CauchyProblemMacro}. Since $\rho_h^s$ are entropy solutions to \eqref{eq:LWRmodel}, we obtain by standard doubling of variable arguments
	\begin{align*}
		-\iint_{(0,\infty)_t\times\Omega_h }|\rho_h^1-\rho_h^2|\,\partial_t\phi_h+G_h(\rho_h^1,\rho_h^2)\,\partial_x\phi_h\d t\d t
		-\int_{\Omega_h}|\rho_{0,h}^1-\rho_{0,h}^2|\,\phi_h(0,x)\d x\le 0
	\end{align*}
	for every $\phi_h \in\mathcal D(\mathbb R_t\times\Omega_h)$, $\phi_h\ge 0$. 
	The solutions $\rho_h^s$ satisfy the coupling condition $\mathcal{RP}(\rho(t,0))=\rho(t,0)$ for a.e.~$t>0$ and therefore also the $L^1$--contraction property proven in Proposition~\ref{prop:ContractionPropertyJunction} and Proposition \ref{prop:EntropyCondToRiemannSolver}. An approximation by test functions $\phi_h$ which are not necessarily zero at the junction leads to
	\begin{align*}
		-\sum_{h=1}^{n+m}\iint_{(0,\infty)_t\times\Omega_h }|\rho_h^1-\rho_h^2|\,\partial_t\phi_h+G_h(\rho_h^1,\rho_h^2)\,\partial_x\phi_h\d t\d x
		-\sum_{h=1}^{n+m}\int_{\Omega_h}|\rho_{0,h}^1-\rho_{0,h}^2|\,\phi_h(0,x)\d x\le 0
	\end{align*}
	for every $\phi_h \in\mathcal D(\mathbb R_t\times\mathbb R_x)$, $\phi_h\ge 0$ with $\phi_h(t,0)=\phi_{h'}(t,0)$ for all $h,h'=1,\dots,n+m$. The contraction property follows by a smooth approximation of $\phi_h(0,x)=\mathbf{1}_{\{t\le T\}}(t,x)$. The uniqueness follows as in \cite{Kr1970}. 
\end{proof}

Theorem~\ref{thm:MainResult} follows from the results in this section.

\section{Numerical Examples for Different Choices of \texorpdfstring{$\hat\eta_h$}{the Entropies}}

We study the solution at the junction with different entropies $\hat\eta_h$. As first example \eqref{eq:SameEntropyHatEta} the same strictly convex entropy $\hat\eta$ for every road $\hat\eta_h=\hat\eta$ is chosen. 
The coupling condition coincides with the one obtained by a vanishing viscosity approach \cite{ACD2017,CoGa2010} and it is independent of the shape of $\hat\eta$.

As second example \eqref{eq:FluxEntropyHatEta} the negative flux functions
	$\hat\eta_h=-f_h$ are considered. The coupling condition can be interpreted such that the drivers maximize their velocity at the junction.

For constant initial data, $n=m=2$ and $K_h=[0,1]$, we consider piece-wise linear flux functions $f_h\colon [0,1]\to [0,1]$
\begin{align*}
	f_1(v)=f_2(v)&=
	\begin{cases}
		2 v,		\quad & 0\le v\le 1/2,\\
		2-2v,	\quad & 1/2\le v\le 1,
	\end{cases}\\
	f_3(v)&=
	\begin{cases}
		4 v	,	\quad & 0\le v\le 1/4,\\
		4/3-4v/3,	\quad &  1/4\le v\le 1,
	\end{cases}\\
	f_4(v)&=
	\begin{cases}
		4v/3,		\quad & 0\le v\le 3/4,\\
		4-4v,	\quad & 3/4\le v\le 1,
	\end{cases}
\end{align*}
illustrated in Figure~\ref{fig:ExampleFluxes}.
The initial data is chosen such that the characteristics speeds are positive:
\begin{align}\label{eq:ExampleInitialData1}
	\rho_{0,i}=1/4\in \Gamma_i^\mathrm{in},i=1,2\qquad\text{and}\qquad
	\rho_{0,j}=0\in \Gamma_j^\mathrm{out},j=3,4.
\end{align}
Solving the Riemann problem with \eqref{eq:SameEntropyHatEta} and \eqref{eq:FluxEntropyHatEta} leads then to qualitatively different solutions. The boundary traces at the junction are given in Table~\ref{tab:TracesRiemannProblemExample1T1}.

\begin{figure}
\begin{center}
	\includegraphics[width=0.35\textwidth]{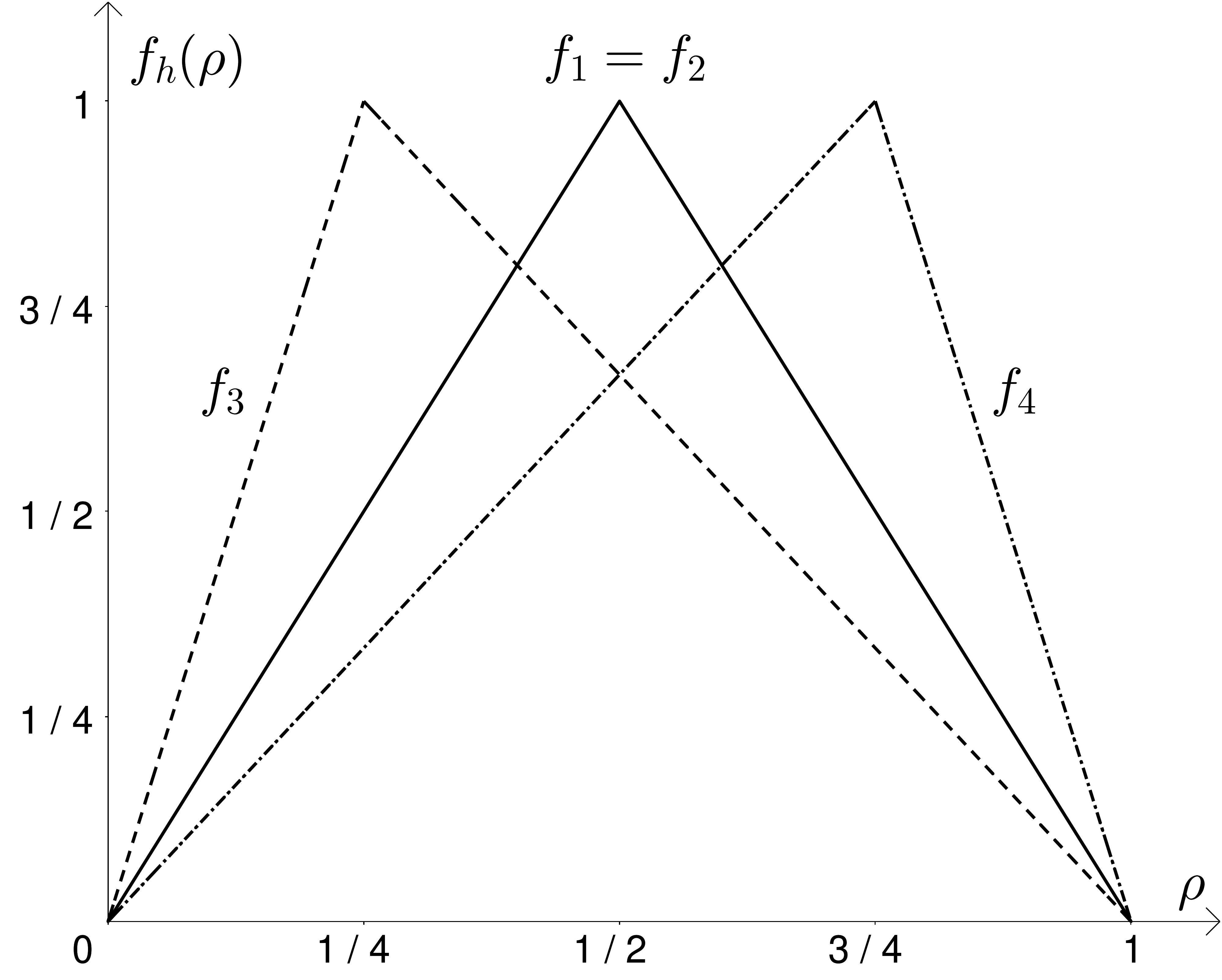}
	\vspace{-0.2cm}
	\caption{The flux functions $f_1,\dots,f_4$.\label{fig:ExampleFluxes}}
\end{center}
\end{figure}

\subsection{Interpretation}
Since we consider free flow initial data, we expect to have one degree of freedom for each outgoing road $j=3,4$. These degrees of freedom are specified by the coupling condition.

In the first example \eqref{eq:SameEntropyHatEta}, there is no prioritization of one of the outgoing roads $j=3,4$ since the entropies $\hat\eta_h$ coincide. Therefore, the number of outgoing cars per unit $\rho_j(t,0)$ is the same for $j=3,4$. 
This example could be used to model traffic junctions where the ratio between the number of cars turning into each outgoing road can be controlled by e.g. traffic lights.

In the example \eqref{eq:FluxEntropyHatEta} we observe that the third road is prioritized since
\begin{equation*}
	\hat\eta_3'(\rho)=-f_3'(\rho)=-4<-4/3=-f_4'(\rho)=\hat\eta_4'(\rho)	\qquad\text{for sufficiently small }\rho>0.
\end{equation*}
Therefore, all cars turn into the third road. The condition may be used to model traffic junctions where drivers choose the outgoing road with less traffic. The drivers choose then the outgoing road such that the distance to the next car in the road is the largest.

\begin{table}
\begin{center}	
	\begin{tabular}{|c|c|c|c|}
	\multicolumn{4}{c}{$\hat\eta_h=\hat\eta$}\\
	\hline
			$\ h\ $	&	$\ \rho_h(t,0)\ $  &	$\ f_h(\rho_h(t,0))\ $ & $\hat\rho_h$\\
	\hline
			$1$	&	$1/4$	&	$1/2$ &	$\ 3/16\ $\\
	\hline
			$2$	&	$1/4$	&	$1/2$ &	$3/16$\\
	\hline
			$3$	&	$3/16$	&	$3/4$ &	$3/16$\\
	\hline
			$4$	&	$3/16$	&	$1/4$ &	$3/16$\\	
	\hline
	\end{tabular}
	\qquad\qquad
	\begin{tabular}{|c|c|c|c|}
	\multicolumn{4}{c}{$\hat\eta_h=-f_h$}\\
	\hline
			$\ h\ $	&		$\ \rho_h(t,0)\ $ &	$\ f_h(\rho_h(t,0))\ $ & $\hat\rho_h$ \\
	\hline
			$1$	&		$1/4$	&	$1/2$ & $0$\\
	\hline
			$2$	&		$1/4$	&	$1/2$ & $0$\\
	\hline
			$3$	&		$1/4$	&	$1$ & $\ 1/4\ $ \\
	\hline
			$4$	&		$0$	&	$0$ & $0$\\	
	\hline
	\end{tabular}
	\vspace{0.2cm}
	\caption{Boundary traces of the solutions to the generalized Riemann problem with initial data \eqref{eq:ExampleInitialData1} and different $\hat\eta_h$. The density $\hat\rho_h$ is as in Definition~\ref{def:RiemannSolver}. \label{tab:TracesRiemannProblemExample1T1}}
\end{center}
\end{table}

%
%
%

\section{Outlook}

\subsection{Drivers Preferences at the Junction}

In the presented considerations we assume that the drives do not have a priori known preferences for the outgoing roads. Coupling conditions with known drivers preferences were intensively studied, see e.g. \cite{GHP2016} and the references therein. 

Typically, instead of using \eqref{eq:MassConservationJunction}, a stronger condition is used: coefficients $\alpha_{ij}\in [0,1]$ satisfying  
	\begin{equation}
		\sum_{j=n+1}^{n+m} \alpha_{ij}=1\quad\text{for }i=1,\dots,n,
	\end{equation}
	are given and the condition
	\begin{equation}\label{eq:TrafficDistributionMacro}
		\sum_{i=1}^n \alpha_{ij}\, f_i(\rho_i(t,0))=f_j(\rho_j(t,0))\qquad \text{for a.e. }t>0,
	\end{equation}
	is introduced. 
	Condition \eqref{eq:TrafficDistributionMacro} can be easily included in Definition~\ref{def:MacroMaximumEntropyProb} and its kinetic version is then given by 	
	\begin{align}
	\begin{split}
		&\sum_{i=1}^{n}\alpha_{ij}\left( \int_{\Gamma_i^{\mathrm{out}}}f'_i(\xi)\, \Psi_i[g(\Gamma^{\mathrm{in}})](\xi)\d{\xi}+\int_{\Gamma_i^{\mathrm{in}}}f'_i(\xi)\, g_i(\xi)\d{\xi}\right)\\
		&= \int_{\Gamma_j^{\mathrm{out}}}f'_j(\xi)\, \Psi_j[g(\Gamma^{\mathrm{in}})](\xi)\d{\xi}+\int_{\Gamma_j^{\mathrm{in}}}f'_j(\xi)\, g_j(\xi)\d{\xi},\quad\text{for } j=n+1,\dots,n+m.
	\end{split}
	\end{align}
	The choice $\hat\eta_h=\operatorname{Id}$ gives the condition in \cite{CGP2005}. It would be interesting to see if other conditions \cite{BrNo2017,DGP2018} can also be obtained by using suitable entropies.	We expect that the obtained new conditions however do not satisfy the $L^1$--contraction property in the sense of Proposition~\ref{prop:ContractionPropertyJunction} and Proposition~\ref{prop:EntropyInequalityJunctionWithout}.

\subsection{\texorpdfstring{$\bf L^1$}{L1}--dissipative solver}
The notion of $L^1$--dissipative solvers is a concept originating from scalar conservation laws with discontinuous flux \cite{AKR2011}. The concept was extended to a network in \cite{ACD2017}. $L^1$--dissipative solvers on networks are  those Riemann solvers satisfying the junction condition in Proposition~\ref{prop:ContractionPropertyJunction}.

In the classical setting of discontinuous fluxes, i.e.~$n=m=1$, the $L^1$--dissipative solvers are characterized by
 $(A,B)\in \Gamma_1^\mathrm{out}\times \Gamma_2^\mathrm{out}$ with $f_1(A)=f_2(B)=s_{(A,B)}$, see \cite{AMG2005,BKT2009}. They are defined by: $\rho=\mathcal{RP}(\rho)$ if and only if 
 \begin{equation}\label{eq:HelpDiscontFluxes1}
 	f_1(\rho_1)=f_2(\rho_2)\le s_{(A,B)}\quad\text{and}\quad
 	\sgn_+(\rho_1-A)\,\sgn_-(\rho_2-B)\le 0.
 \end{equation}
 
Definition~\ref{def:RiemannSolver} and \eqref{eq:HelpDiscontFluxes1} are related in the following way:
If there exists $M\in\mathbb R$, $z\in [0,1]$ such that
	\begin{equation}
		\hat\rho_h=\Pi^M_h(z)\in \Gamma_h^\mathrm{out}, i=1,2\quad\text{and}\quad f_1(\hat\rho_1)=f_2(\hat\rho_2),
	\end{equation}
then Definition~\ref{def:RiemannSolver} and \eqref{eq:HelpDiscontFluxes1} coincide with $s_{(A,B)}=f_1(\hat\rho_1)=f_2(\hat\rho_2)$, $A=\hat\rho_1$ and $B=\hat\rho_2$. Notice that such $M\in\mathbb R$ and $z\in [0,1]$ do not exist if \ref{eq:Assumptionsetah2} holds.
If such $M\in\mathbb R$, $z\in [0,1]$ do not exist, Definition~\ref{def:RiemannSolver} and \eqref{eq:HelpDiscontFluxes1} coincide with $s_{(A,B)}=\max\{f_1^{\max},f_2^{\max}\}$.

Therefore, Theorem~\ref{thm:MainResult} gives an existence and uniqueness result for the LWR model with discontinuous flux for all $L^1$--dissipative solvers.

It would be interesting to see if \eqref{eq:HelpDiscontFluxes1} and the characterisation of all $L^1$--dissipative solvers can be extended to networks. The new formulation has to take into account the prioritization at the junction which is a phenomena that does not occur for $n=m=1$. Inspired by \eqref{eq:HelpDiscontFluxes1} and for notational convenience, we may replace  
	\begin{equation}\label{eq:HelpL1DissipativSolverGeneralEta}
		\hat\eta_h'(\hat\rho_h-)\le M\le \hat\eta_h'(\hat\rho_h+)\quad\text{and}\quad 
		\hat\rho_h=\Pi^M_h(z)\quad\text{for }M\in\mathbb R,z\in [0,1],
	\end{equation}
by a general, more explicit condition on $\hat \rho_h$. 
	We define a function $\Xi\colon [0,1]\to \bigtimes_{h=1}^{n+m} K_h$ and require
	\begin{equation}\label{eq:HelpL1DissipativSolverGeneralXi}
		\hat\rho_h=\Xi_h(y)\qquad\text{for a suitable }y\in [0,1].
	\end{equation} 
	Note that by the intermediate value argument in Proposition~\ref{Prop:ExistenceRiemannProblem}, an increasing, surjective function $\Xi\colon [0,1]\to \bigtimes_{h=1}^{n+m} K_h$ as in \eqref{eq:HelpL1DissipativSolverGeneralXi} can be constructed from \eqref{eq:HelpL1DissipativSolverGeneralEta}.

\subsection*{Acknowledgements}
	The author would like to thank Michael Herty and Michael Westdickenberg for fruitful discussions and for helpful suggestions to improve the manuscript.

\bibliographystyle{plain}
\bibliography{99_Bibliography}

\end{document}